\documentclass[preprint,10pt]{amsart}
\usepackage{lineno,hyperref}
\modulolinenumbers[5]

\usepackage{amsfonts,amscd}
\usepackage{graphicx}
\usepackage{amsmath,mathtools,latexsym,amssymb,amsthm}
\newtheorem{Thm}{Theorem}[section]
\newtheorem{Lem}[Thm]{Lemma}
\newtheorem{Prop}[Thm]{Proposition}
\newtheorem{Def}[Thm]{Definition}
\newtheorem{Exm}[Thm]{Example}
\newtheorem{Rem}[Thm]{Remark}

\usepackage{fancyhdr}

\begin{document}


\title[Universal Central Extensions of Hom-Lie-Rinehart algebras ]{Universal Central Extensions and Non-abelian tensor product of Hom-Lie-Rinehart Algebras}

\author{Ashis Mandal and Satyendra Kumar Mishra}
\footnote{The research of S.K. Mishra is supported by CSIR-SPM fellowship grant 2013 .}

\address{Department of Mathematics and Statistics, Indian Institute of Technology,
Kanpur 208016, India.}




\begin{abstract}
In this paper we study universal central extensions and non-abelian tensor product of hom-Lie-Rinehart algebras. We define universal $\alpha$- central extensions, and discuss about lifting of automorphisms and $\alpha$-derivations to central extensions of hom-Lie-Rinehart algebras. This in turn provides such a lifting of automorphisms and $\alpha^k$-derivations to  central extensions in the particular case of hom-Lie algebras. 

\end{abstract}




\maketitle
\section{Introduction}
The twisted algebraic structures or hom-algebraic structures ( hom-associative, hom-Lie, hom-Leibniz and so on ) are a topic of growing interest in many branches of Mathematics and Mathematical Physics.  In \cite{HLS06},  J. Hartwig and coauthors introduced the notion of hom-Lie algebras in the context of some particular deformation called  $q$-deformation of Witt and Virasoro algebra of vector fields. They used this new type of algebras to describe the $q$-deformations using  $\sigma$-derivations in place of usual derivations. Subsequently, many concepts and properties have been investigated not only in the context of hom-Lie algebras but also for  hom-algebra structures corresponding to several classical algebras as well. This includes the works of J. Hartwig, D. Larsson, A. Makhlouf, S. Silvestrov, D. Yau and other authors (\cite{AEM11}, \cite{HLS06}, \cite{MS08}, \cite{MS10},\cite{Yau09}). It is important to note here that such investigations are not bound to be straight forward generalisations from their counterparts in the classical algebra case.

The notion of hom-Lie-Rinehart algebra is studied in \cite{MM2018,MM2017}, as a generalisation of hom-Lie algebras and an algebraic analogue of a hom-Lie algebroid (see \cite{LGT13}). A cohomology with coefficients in a left module is introduced and it has been interpreted in lower dimensions to consider the extensions of hom-Lie-Rinehart algebras, which also gives a characterisation of certain abelian and central extensions of hom-Lie algebras. In \cite{MM2017}, a homology of hom-Lie-Rinehart algebras with coefficients in right modules is defined, which generalises the hom-Lie algebra homology, studied in \cite{Yau09}. 

In the case of Lie algebras, the universal central extensions are studied in \cite{Gar80, Kallen}, and a characterisation of kernel of the universal central extension is given in terms of second homology group with trivial coefficients. These results are generalised for the case of Lie-Rinehart algebras in \cite{CGML14}.  On the other hand, in order to proceed for a hom-Lie algebra, in  \cite{CIP15} it is shown that composition of two central extensions may not be a central extension. This observation leads to a new notion of the $\alpha$-central extension of a hom-Lie algebra. The authors proved that if a hom-Lie algebra is perfect, then it has a universal central extension. They also presented an existence of a universal $\alpha$-central extension for a hom-Lie algebra which is $\alpha$-perfect. Thus, it is evident that the classical results in the theory of universal central extensions of Lie algebras do not simply follow for hom-Lie algebras,  because of the absence of the fact that the composition of two central extensions is a central extension (which holds true in the categories of groups, Lie algebras, Leibniz algebras, Lie-Rinehart algebras, etc.). Here, we develop a theory of universal central extensions for hom-Lie-Rinehart algebras and discuss the lifting of automorphisms and $\alpha$-derivations of hom-Lie-Rinehart algebras to central extensions. In this process, one can also deduce such lifting of automorphisms and $\alpha^k$-derivations  (defined in \cite {Sheng12}) to the central extensions of hom-Lie algebras. 
 
Moreover, a non-abelian tensor product of Lie algebra is introduced in \cite{Ellis91} and it is related to universal central extensions of Lie algebras. Later on, this study is extended to Leibniz algebras, Lie-Rinehart algebras, hom-Lie algebras and hom-Leibniz algebras in \cite{Gne99}, \cite{CGML14}, \cite{CKP17}, and \cite{CKP16} respectively. In this paper, we define a non-abelian tensor product for hom-Lie-Rinehart algebras and discuss its properties. We also express universal central extensions and universal $\alpha$- central extensions of hom-Lie-Rinehart algebras in terms of the non-abelian tensor product.   

In section $2$, we recall some basic definitions and terminologies for hom-Lie-Rinehart algebras. Then in section $3$, we discuss about univeral central extensions of a hom-Lie-Rinehart algebra. We also prove that for a perfect hom-Lie-Rinehart algebra there exists a universal central extension. In section $4$,  we consider $\alpha$- central extensions and universal $\alpha$- central extensions of hom-Lie-Rinehart algebras. The lifting of automorphisms and $\alpha$-derivations are discussed in section $5$. Finally, in section $6$, we present a discussion of non-abelian tensor product of hom-Lie-Rinehart algebras. 

\section{Preliminaries on hom-Lie-Rinehart algebras}
In this section, we recall basic definitions concerning hom-algebra structures from the literature (\cite{AEM11},\cite{HLS06}, \cite{MS08}, \cite{MS10}, \cite{Sheng12}) in order to  fix notation and terminology will be needed throughout the paper. Let $R$ denote a commutative ring with unity. We will consider all modules, algebras and their tensor products over such a ring $R$ and all linear maps to be $R$-linear unless otherwise stated. 

A {\bf hom-Lie algebra} is a triplet $(\mathfrak{g},[-,-],\alpha)$ where $\mathfrak{g}$ is an $R$-module equipped with a skew-symmetric $R$-bilinear map $[-,-]:\mathfrak{g}\times \mathfrak{g}\rightarrow \mathfrak{g}$ and a linear map $\alpha:\mathfrak{g}\rightarrow \mathfrak{g}$, satisfying $\alpha[x,y]=[\alpha(x),\alpha(y)]$ such that the hom-Jacobi identity holds:
\begin{equation*}
[\alpha(x),[y,z]]+[\alpha(y),[z,x]]+[\alpha(z),[x,y]]=0 ~~~~\mbox{for all}~~x,y,z\in \mathfrak{g}.
\end{equation*}
Furthermore, if $\alpha$ is an automorphism of the $R$-module $\mathfrak{g}$, then the hom-Lie algebra $(\mathfrak{g},[-,-],\alpha)$ is called a {\bf regular hom-Lie algebra}.

A {\bf representation of a hom-Lie algebra} $(\mathfrak{g},[-,-],\alpha)$ on a $R$-module $V$ is a pair $(\theta, \beta)$ of $R$-linear maps $\theta: \mathfrak{g}\rightarrow \mathfrak{g}l(V)$ and $\beta: V\rightarrow V$ such that 
\begin{align*}
\theta(\alpha(x))\circ \beta&= \beta\circ \theta(x),\\
\theta([x,y])\circ \beta&= \theta(\alpha(x))\circ\theta(y)-\theta(\alpha(y))\circ\theta(x).
\end{align*}
for all $x,y \in \mathfrak{g}$.
Given an associative commutative algebra $A$, an $A$-module $M$ and an algebra endomorphism $\phi: A\rightarrow A$, we call an $R$-linear map $\delta: A\rightarrow M$ {\bf a $\phi$-derivation of $A$ into $M$} if it satisfies the required $\phi$-derivation rule:
$$\delta(a b)=\phi(a) \delta(b)+\phi(b) \delta(a)~~\mbox{for all}~~a,b \in A.$$
 Let us denote by $Der_{\phi}(A)$ the set of all $\phi$-derivations $\delta: A\rightarrow A$.
\subsection{Hom-Lie-Rinehart algebras}
\begin{Def}\label{hom-LR}
A hom-Lie Rinehart algebra over $(A,\phi)$  is defined as a tuple $(A,L,[-,-]_L,\phi,\alpha_L,\rho_L)$,
 where $A$ is an associative commutative algebra, $L$ is an $A$-module, $[-,-]_L: L\times L\rightarrow L$ is a skew symmetric bilinear map, the map $\phi:A\rightarrow A$ is an algebra homomorphism, $\alpha_L :L\rightarrow L $ is a linear map satisfying $\alpha_L([x,y]_L)=[\alpha_L(x),\alpha_L(y)]_L$, and the $R$-linear map $\rho_L: L\rightarrow Der_{\phi}A$  are such that following conditions hold.
\begin{enumerate}
\item The triplet $(L,[-,-]_L,\alpha_L)$ is a hom-Lie algebra.
\item $\alpha_L(a.x)=\phi(a).\alpha_L(x)$ for all $a\in A,~x\in L $.
\item $(\rho_L, \phi)$ is a representation of $(L,[-,-]_L,\alpha_L)$ on $A$.
\item $\rho_L(a.x)=\phi(a).\rho_L(x)$ for all $a\in A,~x\in L $.
\item $[x,a.y]_L=\phi(a)[x,y]_L+ \rho_L(x)(a)\alpha_L(y)$ for all $a\in A,~x,y\in L $.
\end{enumerate}
\end{Def}
A  hom-Lie-Rinehart algebra $(A,L,[-,-]_L,\phi,\alpha_L,\rho_L)$ is said to be {\bf regular} if the map $\phi:A\rightarrow A$ is an algebra automorphism and $\alpha_L:L\rightarrow L $ is a bijective linear map. If we take $\alpha_L=Id_L$ in the above definition, then $\phi=Id_A$ and the hom-Lie-Rinehart algebra $(A,L,[-,-]_L,\phi,\alpha_L,\rho_L)$ is a Lie-Rinehart algebra $L$ over $A$. (See \cite{MM2017}).
\begin{Exm}\label{h-LA}
A  hom-Lie algebra $(L,[-,-]_L,\alpha_L)$ structure over an $R$-module $L$ gives the hom-Lie Rinehart algebra $(A,L,[-,-]_L,\phi,\alpha_L,\rho_L)$ with $A=R$, the algebra morphism $\phi=Id_{R}$ and the trivial action of $L$ on $R$. 
\end{Exm}
\begin{Exm}\label{h-Lie-algebroid}
Any hom-Lie algebroid $(A,\phi,[-,-],\rho,\alpha)$ over a smooth manifold $M$, gives a hom-Lie-Rinehart algebra $(C^{\infty}(M),\Gamma A,[-,-],\phi^*,\alpha,\rho)$ over $(C^{\infty}(M),\phi^*)$, where $\Gamma A$ is the space a sections of the underline vector bundle $A$ over  $M$ and the algebra homomorphism $\phi^*:C^{\infty}(M)\rightarrow C^{\infty}(M)$ is induced by the smooth map $\phi: M\rightarrow M$.
\end{Exm}

\begin{Exm}\label{hom-LR byc}
If we consider a Lie-Rinehart algebra $L$ over $A$ along with an endomorphism $$(\phi,\alpha):(A,L)\rightarrow (A,L)$$ in the category of Lie-Rinehart algebras then the tuple $(A,L,[-,-]_{\alpha}, \phi,\alpha,\rho_{\phi})$ is a hom-Lie-Rinehart algebra, called \textbf{``obtained by composition"}, where
\begin{enumerate}
\item $[x,y]_{\alpha}=\alpha[x,y]$ for $x,y\in L$; 
\item $\rho_{\phi}(x)(a)=\phi(\rho(x)(a))$ for $x\in L, ~a\in A$.
\end{enumerate} 
\end{Exm}

\begin{Exm}\label{Product}
Let $(A,L,[-,-]_L,\phi,\alpha_L,\rho_L)$ and $(A,M,[-,-]_M,\phi,\alpha_M,\rho_M)$ be hom-Lie-Rinehart algebras over $(A,\phi)$. We consider $$L\times_{Der_{\phi}A}M=\{(l,m)\in L\times M : \rho_l(l)=\rho_M(m)\},$$ where $L\times M$ denotes the Cartesian product. Then $(A,L\times_{Der_{\phi}A}M,[-,-],\phi,\alpha,\rho)$ is a hom-Lie-Rinehart algebra, where  
\begin{enumerate}
\item the bracket is given by 
$$[(l_1,m_1),(l_2,m_2)]=([l_1,l_2]_L,[m_1,m_2]_M);$$
\item the endomorphism $\alpha:L\times_{Der_{\phi}A}M\rightarrow L\times_{Der_{\phi}A}M$ is given by
$$\alpha(l,m)=(\alpha_L(l),\alpha_M(m));$$
\item and the anchor map $\rho: L\times_{Der_{\phi}A}M\rightarrow Der_{\phi}A$ is given by
$$\rho(l,m)(a)=\rho_L(a)=\rho_M(a);$$
\end{enumerate}   
for all $l,l_1,l_2\in L$, $m,m_1,m_2\in M$, and $a\in A$. The above  structure gives the categorical product in the category $hLR_A^{\phi}$. 
\end{Exm}

A {\bf subalgebra} of a hom-Lie-Rinehart algebra $(\mathcal{L},\alpha_L)$ over $(A,\phi)$ is a pair $(\mathcal{M},\alpha_{M})$ such that underlying $A$-module $M$ is an $A$-submodule of $L$, restriction of the map $\alpha_L$ on $M$ becomes an endomorphism on $M$ $(\alpha_{M}=\alpha_L{\big|_M})$, and $[m,n]_L\in M$ for $m,n \in M$. The pair $(\mathcal{M},\alpha_{|_M})$ is called a {\bf quasi ideal} in $(\mathcal{L},\alpha)$ if $[m,x]\in M$ for $m\in M$, and $x\in L$. If $\rho_L(m)=0$ for each $m\in M$, then this pair $(\mathcal{M},\alpha_L{\big|_M})$ is called an {\bf ideal} in $(\mathcal{L},\alpha)$. 
\begin{Def}
Let $(A,L,[-,-]_{L},\phi,\alpha_L,\rho_L)$ and $(B,L^{\prime},[-,-]_{L^{\prime}},\psi,\alpha_{L^{\prime}},\rho_{L^{\prime}})$ be hom-Lie-Rinehart algebras, then a {\bf homomorphism }of hom-Lie-Rinehart algebras is defined as a pair of maps $(g,f)$, where the map $g:A\rightarrow B$ is a $R$-algebra homomorphism and $f: L_1\rightarrow L_2$ is a $R$-linear map such that following identities hold:
\begin{enumerate}
\item $f(a.x)=g(a).f(x)~~\mbox{for all}~x\in L_1,~a\in A,$
\item $f[x,y]_L=[f(x),f(y)]_{L^{\prime}} ~~\mbox{for all}~x,y\in L,$
\item $f(\alpha_L(x))=\alpha_{L^{\prime}}(f(x)) ~~\mbox{for all}~x\in L,$
\item $g(\phi(a))=\psi(g(a))~~\mbox{for all}~a\in A$
\item $g(\rho_L(x)(a))=\rho_{L^{\prime}}(f(x))(g(a))~~\mbox{for all}~x\in L,~ a\in A.$
\end{enumerate}  
\end{Def} 
Hom-Lie-Rinehart algebras with homomorphisms form a category  of hom-Lie-Rinehart algebras, which we denote by $hLR$. Note that the category of Lie-Rinehart algebras is a full subcategory of the category of hom-Lie-Rinehart algebras. 

\begin{Rem}
If $A=B$ and $\phi=\psi$, then by taking $g=Id_A$, we get homomorphism of hom-Lie-Rinehart algebras over $(A,\phi)$. Let us denote by $hLR^{\phi}_A$ the category of hom-Lie Rinehart algebras over $(A,\phi)$. To simplify the notations we denote a hom-Lie Rinehart algebra $(A,L,[-,-]_L,\phi,\alpha_L,\rho_L)$ over $(A,\phi)$ by $(\mathcal{L},\alpha_L)$ and similarly for any other hom-Lie Rinehart algebra over $(A,\phi)$, say $(A,L^{\prime},[-,-]_{L^{\prime}},\phi,\alpha_{L^{\prime}},\rho_{L^{\prime}})$, denote it simply by the notation $(\mathcal{L}^{\prime},\alpha_{L^{\prime}})$.
\end{Rem}

\begin{Def}\label{Mod}
Let $M$ be an $A$-module, and $\beta\in End_{R}(M)$. Then the pair $(M,\beta)$ is {\bf a left module} over a hom-Lie Rinehart algebra $(\mathcal{L},\alpha_L)$ if the following holds.
\begin{enumerate}
\item There is a map $\theta:L\otimes M\rightarrow M$, such that the pair  $(\theta,\beta)$ is a representation of the hom-Lie algebra $(L,[-,-]_L,\alpha_L)$ on $M$. Let us denote $\theta(x,m)$ by $\{x,m\}$ for $x\in L,~m\in M$.
\item $\beta(a.m)=\phi(a).\beta(m)$ for all $a\in A~\mbox{and}~m\in M$.
\item $\{a.X,m\}=\phi(a)\{X,m\}$ for all $a\in A,~X\in L,~m\in M$.
\item $\{X,a.m\}=\phi(a)\{X,m\}+\rho_L(X)(a).\beta(m)$ for all $X\in L,~a\in A,~m\in M$.
\end{enumerate} 
\end{Def}
In particular, for $\alpha_L=Id_L$ and $\beta=Id_M$, $(\mathcal{L},\alpha_L)$ is a Lie-Rinehart algebra and $M$ is a left Lie-Rinehart algebra module over the Lie-Rinehart algebra $L$.
\begin{Exm}
The pair $(A,\phi)$ is a canonical left $(\mathcal{L},\alpha_L)$-module, where the left action of $L$ on $A$ is given by the anchor map. 
\end{Exm}
Let $(\mathcal{L},\alpha_L)$ be a hom-Lie Rinehart algebra over $(A,\phi)$ and $(M,\beta)$ be a left module over $(\mathcal{L},\alpha_L)$. We consider the graded $R$-module 
$C^*(L;M):=\oplus_{n\geq 1}C^n(L;M)$
 for hom-Lie-Rinehart algebra $(\mathcal{L},\alpha_L)$ with coefficients in $(M,\beta)$, where $C^n(L;M)\subseteq Hom_R(\wedge_R^n L,M)$ consisting of elements $f\in Hom_R(\wedge_R^n L,M)$ satisfying conditions below.
\begin{enumerate}
\item $f(\alpha_L(x_1),\cdots,\alpha_L(x_n))=\beta(f(x_1,x_2,\cdots,x_n))$ for all $x_i\in L,~1\leq i\leq n$
\item $f(x_1,\cdots,a.x_i,\cdots,x_n)=\phi^{n-1}(a)f(x_1,\cdots,x_i,\cdots,x_n)$ for all $x_i\in L,~1\leq i\leq n,~\mbox{and}~ a\in A$.
\end{enumerate}
Define the $R$-linear maps $\delta:C^n(L;M)\rightarrow C^{n+1}(L;M) $ given by 
\begin{equation*}
\begin{split}
&\delta f(x_1,\cdots,x_{n+1})\\
&:= \sum_{i=1}^{n+1}(-1)^{i+1}\{\alpha_L^{n-1}(x_i),f(x_1,\cdots,\hat{x_i},\cdots,x_{n+1})\}\\&+\sum_{1\leq i<j\leq n+1}f([x_i,x_j]_L,\alpha_L(x_1),\cdots,\hat{\alpha_L(x_i)},\cdots,\hat{\alpha_L(x_j)},\cdots,\alpha_L(x_{n+1}))
\end{split}
\end{equation*}
for all $f\in C^n(L;M),~ x_i\in L, $ and $1\leq i\leq n+1$. Here, $(C^*(L,M),\delta)$ forms a a cochain complex, see \cite{MM2018} for more details. 
\begin{Def}
Let $M$ be an $A$-module and $\beta\in End_{R}(M)$. Then the pair $(M,\beta)$ is {\bf a right module} over a hom-Lie Rinehart algebra $(\mathcal{L},\alpha_L)$ if the following conditions hold.
\begin{enumerate}
\item There is a map $\theta:M\otimes L\rightarrow M$ such that the pair  $(\theta,\beta)$ is a representation of the hom-Lie algebra $(L,[-,-]_L,\alpha_L)$ on $M$, where  $\theta(m,x)$ is usually denoted by $\{m,x\}$ for $x\in L,~m\in M$.
\item $\beta(a.m)=\phi(a).\beta(m)$ for $a \in A$ and $m \in M$.
\item $\{a.m,x\}=\{m,a.x\}=\phi(a).\{m,x\}-\rho_L(x)(a).\beta(m)$ for $a\in A,~x\in L,~m\in M$.
\end{enumerate} 
\end{Def}
If $\alpha=Id_L$ and $\beta =Id_M$, then M is a right Lie-Rinehart algebra module. Note that there is no canonical right module structure on $(A,\phi)$ as one would expect from the case of Lie-Rinehart algebras. 




\section{Universal Central Extensions of Hom-Lie-Rinehart Algebras}
 We recall from \cite{MM2018}, some necessary definitions and results about extensions of a hom-Lie-Rinehart algebra. First note that the category $hLR_A^{\phi}$ does not have zero object. Thus, by a {\bf short exact sequence} in the category $hLR_A^{\phi}$:
$$\begin{CD}
(\mathcal{L}^{\prime\prime},\alpha_{L^{\prime\prime}}) @>i>> (\mathcal{L}^{\prime},\alpha_{L^{\prime}}) @>\sigma>> (\mathcal{L},\alpha_L)@.
\end{CD}$$ 
we mean that the homomorphism $i:(\mathcal{L}^{\prime\prime},\alpha_{L^{\prime\prime}})\rightarrow(\mathcal{L}^{\prime},\alpha_{L^{\prime}})$ is injective, the homomorphism $\sigma:(\mathcal{L}^{\prime},\alpha_{L^{\prime}})\rightarrow (\mathcal{L},\alpha_L)$ is surjective, and $ker (\sigma)=Im (i)$. 
\begin{Def} 
A  short exact sequence in the category $hLR^{\phi}_A$
$$\begin{CD}
(\mathcal{L}^{\prime\prime},\alpha_{L^{\prime\prime}}) @>i>> (\mathcal{L}^{\prime},\alpha_{L^{\prime}}) @>\sigma>> (\mathcal{L},\alpha_L)@.
\end{CD}$$ 
is called an {\bf extension} of the hom-Lie-Rinehart algebra $(\mathcal{L},\alpha)$ by the hom-Lie-Rinehart algebra $(\mathcal{L}^{\prime\prime},\alpha_{L^{\prime\prime}})$. Here, the anchor map of the hom-Lie-Rinehart algebra $(\mathcal{L}^{\prime\prime},\alpha_{L^{\prime\prime}})$ is zero, i.e. $\rho_{L^{\prime\prime}}=0$ since $\sigma\circ i=0$.
\end{Def}
An extension 
$$\begin{CD}
(\mathcal{L}^{\prime\prime},\alpha_{L^{\prime\prime}}) @>i>> (\mathcal{L}^{\prime},\alpha_{L^{\prime}}) @>\sigma>> (\mathcal{L},\alpha_L)@.
\end{CD}$$ 
of a hom-Lie-Rinehart algebra $(\mathcal{L},\alpha_L)$ by a hom-Lie-Rinehart algebra $(\mathcal{L}^{\prime\prime},\alpha_{L^{\prime\prime}})$ is said to be {\bf $A$-split} if we have an $A$-module map $\tau: (\mathcal{L},\alpha_L)\rightarrow (\mathcal{L}^{\prime},\alpha_{L^{\prime}})$ such that
\begin{enumerate}
\item $\sigma \circ \tau= Id_L,$
\item $\tau(a.x)=a.\tau(x)$ for each $a\in A,~x\in L$, and
\item $\tau\circ\alpha_L = \alpha_{L^{\prime}}\circ \tau.$
\end{enumerate}
In this case, we call $\tau$ is a section of the map $\sigma$. Furthermore if the section $\tau$ of the map $\sigma$ is a homomorphism of hom-Lie-Rinehart algebras, then the given extension is said to be {\bf split} in the category of hom-Lie-Rinehart algebras.

First let us observe that any hom-Lie-Rinehart algebra module $(M,\beta)$ gives a hom-Lie-Rinehart algebra $(A,M,[-,-]_M,\phi,\beta,\rho_M) \in hLR^{\phi}_A$, with a trivial bracket and a trivial anchor map. We denote this object in $hLR^{\phi}_A$ by $(\mathcal{M},\beta)$. Now recall an abelian extension of hom-Lie-Rinehart algebras from \cite{MM2018}. 

\begin{Def}
Let $(\mathcal{L},\alpha_L)$ be a hom-Lie-Rinehart algebra over $(A,\phi)$ and $(M,\beta)$ be a left $(\mathcal{L},\alpha_L)$-module. Then a short exact sequence 
$$\begin{CD}
(\mathcal{M},\beta)@>i>> (\mathcal{L}^{\prime},\alpha_{L^{\prime}}) @> \epsilon>>(\mathcal{L},\alpha_L)@. 
\end{CD}$$
in the category $hLR^{\phi}_A$, is called an {\bf abelian extension} of $(\mathcal{L},\alpha_L)$ by $(M,\beta)$ if 
$$[i(m),x]_{L^{\prime}}=i((\epsilon(x)).m)~~\mbox{for all }m\in M, ~x\in L^{\prime}.$$
\end{Def}
In \cite{MM2018}, it is proved that the second cohomology space $H^2_{hLR}(L,M)$ of a hom-Lie-Rinehart algebra $(\mathcal{L},\alpha_L)$ with coefficients in $(M,\beta)$ classifies $A$-split abelian extensions of $(\mathcal{L},\alpha_L)$ by $(M,\beta)$. In particular, the following result generalises the well-known classification theorems for the classical cases of a Lie algebras \cite{HilSta71} and  Lie-Rinehart algebras \cite{CGML14}.  
\begin{Thm}\cite{MM2018}\label{Char1}
 There is a one-to-one correspondence between the equivalence classes of $A$-split abelian extensions of a hom-Lie-Rinehart algebra $(\mathcal{L},\alpha)$ by $(M,\beta)$ and the cohomology classes in $H^2_{hLR}(L,M)$.
\end{Thm}
Let us recall that the center of a hom-Lie-Rinehart algebra $(\mathcal{L},\alpha_L)$ given by
$$Z_A (\mathcal{L})=\{x\in L : [a.x,z]_L=[a.\alpha_L(x),z]_L=0,~\mbox{and} ~x(a)=0~~\mbox{for all}~a\in A,~z\in L \}.$$
It is an ideal of the hom-Lie-Rinehart algebra $(\mathcal{L},\alpha_L)$.
 \begin{Def}
A short exact sequence of hom-Lie-Rinehart algebras 
$$\begin{CD}
(\mathcal{M},\beta)@>i>> (\mathcal{L}^{\prime},\alpha_{L^{\prime}}) @> \sigma>>(\mathcal{L},\alpha_L)
\end{CD}$$
is called a central extension of $(\mathcal{L},\alpha_L)$ if $i(M)=Ker(\sigma)\subset Z_A(\mathcal{L}^{\prime}) $. Here, $(\mathcal{M},\beta)$ is a hom-Lie-Rinehart algebra $(A,M,[-,-]_M,\phi,\beta,\rho_M)$. 
\end{Def}
\begin{Rem}\label{Cent}
\begin{enumerate}
\item Since $\sigma\circ i=0$, we have $\rho_M=0$.
\item Note that $i(M)=Ker(\sigma)\subset Z_A(\mathcal{L}^{\prime})$ implies $(\mathcal{M},\beta)$ is a hom-Lie-Rinehart algebra with trivial bracket. Since $i[m,n]_M=[i(m),i(n)]_{L^\prime}=0$ and $i$ is an injective
map.
\item If $(M,\beta)$ is a trivial hom-Lie-Rinehart algebra module over $(\mathcal{L},\alpha_L)$, then an abelian extension of $(\mathcal{L},\alpha_L)$ by the module $(M,\beta)$ is a central extension.  
\end{enumerate}
\end{Rem}

\begin{Prop}\cite{MM2018}\label{EQC}
There is a one-to-one correspondence between the equivalence classes of $A$-split central extensions
$$\begin{CD}
(\mathcal{M},\beta)@>i>> (\mathcal{L}^{\prime},\alpha_{L^{\prime}}) @> \sigma>>(\mathcal{L},\alpha_L)
\end{CD}$$
of $(\mathcal{L},\alpha_L)$ by $(\mathcal{M},\beta):=(A,M,[-,-]_M,\phi,\beta,\rho_M)$ and the cohomology classes in $H^2_{hLR}(L,M)$, where $(M,\beta)$ is a trivial module over the hom-Lie-Rinehart algebra $(\mathcal{L},\alpha_L)$.
\end{Prop}
\begin{Def}
A central extension 
$$\begin{CD}
(\mathcal{M},\alpha_M)@>i>> (\mathcal{K},\alpha_K) @> \sigma>>(\mathcal{L},\alpha_L)
\end{CD}$$
is said to be {\bf universal central extension} if for any other central extension
$$\begin{CD}
(\mathcal{M}^{\prime},\alpha_{M^{\prime}})@>j>> (\mathcal{L}^{\prime},\alpha_{L^{\prime}}) @> \tau>>(\mathcal{L},\alpha_L)
\end{CD}$$
there exists a unique homomorphism $h:(\mathcal{K},\alpha_K)\rightarrow (\mathcal{L}^{\prime},\alpha_{L^{\prime}})$ in the category $hLR_A^{\phi}$ such that $\tau\circ h=\sigma$.
\end{Def}
\begin{Thm}\label{Characterisation}
Let  $(\mathcal{M},\alpha_M)\xrightarrow{i} (\mathcal{K},\alpha_K) \xrightarrow{\sigma}(\mathcal{L},\alpha_L)$ be a central extension of a hom-Lie-Rinehart algebra $(\mathcal{L},\alpha_L)$. If every central extension of $(\mathcal{K},\alpha_K)$ splits uniquely in the category $hLR_A^{\phi}$, then the extension $(\mathcal{M},\alpha_M)\xrightarrow{i} (\mathcal{K},\alpha_K) \xrightarrow{\sigma}(\mathcal{L},\alpha_L)$ is a universal central extension.
\end{Thm}
\begin{proof}
Let $(\mathcal{M}^{\prime},\alpha_{M^{\prime}})\xrightarrow{j} (\mathcal{L}^\prime,\alpha_{L^\prime})\xrightarrow{\tau} (\mathcal{L},\alpha_L)$ be a central extension of $(\mathcal{L},\alpha_L)$. We consider a hom-Lie-Rinehart algebra for the pull-back diagram corresponding to the homomorphisms $\sigma:(\mathcal{K},\alpha_K) \rightarrow(\mathcal{L},\alpha_L)$, and $\tau:(\mathcal{L}^\prime,\alpha_{\mathcal{L}^\prime}) \rightarrow(\mathcal{L},\alpha_L)$ as follows: define an $A$-module $P=\{(k,l^\prime): K\times \mathcal{L}^\prime: \sigma(k)=\tau(l^\prime)\}$ and the maps $ \alpha_P(k,l^\prime)= (\alpha_K(k), \alpha_{\mathcal{L}^\prime}(l^\prime)) $; $\rho_P (k,l^\prime)(a)= \rho_L(\sigma(k))(a) = \rho_L(\tau(l^\prime))(a) $. Then $(\mathcal{P},\alpha_P)=(A,P,[-,-]_P,\phi,\alpha_P,\rho_P)$ is a hom-Lie-Rinehart algebra (as the product structure is defined in Example \ref{Product}).

Let the map $\pi_1:P\rightarrow K$ denotes  the projection map onto the first factor $K$. Now the short exact sequence
$  (ker(\pi_1),\alpha_P |_{ker(\pi_1)})\hookrightarrow (\mathcal{P},\alpha_P)\xrightarrow{\pi_1} (\mathcal{K},\alpha_K)$
is a central extension of $(\mathcal{K},\alpha_K)$. So, it splits uniquely. Assume that $s:(\mathcal{K},\alpha_K)\rightarrow (\mathcal{P},\alpha_P)$ is the unique section of $\pi_1$, i.e. $ \pi_1 \circ s=Id_K$, then for any $k\in K$, we have $s(k)=(k, l^\prime)$ for some $l^\prime \in L^\prime$. Thus the section $s$ of the map $ \pi_1$ induces a map $h:K\rightarrow L^\prime$. Since $s$ is a homomorphism in $hLR_A^{\phi}$, it follows that $h:(\mathcal{K},\alpha_K)\rightarrow (\mathcal{L}^\prime,\alpha_{L^\prime})$ is a morphism in $hLR_A^{\phi}$ such that $\beta\circ  h=\sigma$. The uniqueness of the homomorphism $h$ follows from the uniqueness of the section $s:(\mathcal{K},\alpha_K)\rightarrow (\mathcal{P},\alpha_P)$.

\end{proof}

\subsection{Universal central extensions for perfect hom-Lie-Rinehart algebras}
\begin{Def}
A hom-Lie-Rinehart algebra $(\mathcal{L},\alpha_L)$ is said to be {\bf perfect} if $L$ is the $A$-submodule $\{L,L\}$, generated by the elements of the form $[x,y]_L$ for all $x,y\in L$. 
\end{Def}

In particular, if $\alpha_L=Id_L$, then a perfect hom-Lie-Rinehart algebra is simply a perfect Lie-Rinehart algebra defined in \cite{CGML14}.  Moreover, if $A=R$, and $\phi=Id_A$, then a perfect hom-Lie-Rinehart algebra is simply a perfect hom-Lie algebra defined in \cite{CIP15}.
\begin{Rem}
Let $\sigma:(\mathcal{K},\alpha_K) \rightarrow(\mathcal{L},\alpha_L)$ be a surjective morphism in the category $hLR^{\phi}_A$ and $(\mathcal{K},\alpha_K)$ is a perfect hom-Lie-Rinehart algebra. Then $(\mathcal{L},\alpha_L)$ is also a perfect hom-Lie-Rinehart algebra.
\end{Rem}
\begin{Lem}\label{uniq}
Let  $
(\mathcal{M}^{\prime},\alpha_{M^{\prime}})\xrightarrow{j} (\mathcal{L}^{\prime},\alpha_{L^{\prime}}) \xrightarrow{\tau}(\mathcal{L},\alpha_L)$ be a central extension of $(\mathcal{L},\alpha_L)$ and $(\mathcal{K},\alpha_K)$ be a perfect hom-Lie-Rinehart algebra. If there exist hom-Lie-Rinehart algebra homomorphisms $f,g:(\mathcal{K},\alpha_K)\rightarrow (\mathcal{L}^{\prime},\alpha_{L^{\prime}})$ such that $\tau\circ f=\tau\circ g$, then $f=g$.
\end{Lem}
\begin{proof}
 Since $(\mathcal{K},\alpha_K)$ is a perfect hom-Lie-Rinehart algebra the proof follows if we can show that $g(a[x,y])=f(a[x,y])$ for all $a \in A$ and $x,y \in K$.

Note that $\tau(x_1)=\tau(x_2)$ and $\tau(y_1)=\tau(y_2)$, then $x_1-x_2, y_1-y_2\in ker(\tau)$ and since $ker(\tau)\subset Z_A(\mathcal{L}^{\prime},\alpha_{L^{\prime}})$, it follows that $[x_1,y_1]=[x_2,y_2]$. Now, if we take $x_1=f(x), y_1=f(y)$ and $x_2=g(x), y_2=g(y)$ it follows that $[f(x),f(y)]= [g(x),g(y)]$. Therefore, we have  $g(a[x,y])=a[g(x),g(y)]=a[f(x),f(y)]=f(a[x,y])$. This completes the proof.
\end{proof}

Next, we construct a hom-Lie-Rinehart algebra $(A,\mathfrak{uce}^{\phi}_A L,[-,-],\phi,\tilde{\alpha_L},\tilde{\rho_L})$ for a given hom-Lie-Rinehart algebra $(\mathcal{L},\alpha_L)$ over $(A,\phi)$.  Consequently, this construction gives a functor $\mathfrak{uce}_A^{\phi}:hLR_A^{\phi}\rightarrow hLR_A^{\phi}$,  which associates a universal central extension to a perfect hom-Lie-Rinehart algebra.
 
First we consider an $A$-submodule $M^{\phi}_A L$ of the $A$-module $A\otimes L\otimes L$, which is generated by the elements of the following forms: 
\begin{enumerate}
\item $a\otimes x\otimes x$;
\item $a\otimes x\otimes y+a\otimes y\otimes x$;
\item $a\otimes \alpha_L(x)\otimes [y,z]_L+a\otimes \alpha_L(y)\otimes [z,x]_L+a\otimes \alpha_L(z)\otimes [x,y]_L$; 
\item $\phi(a)\otimes [x,y]_L\otimes [x^{\prime},y^{\prime}]_L+\rho([x,y]_L)(a)\otimes \alpha_L(x^{\prime})\otimes \alpha_L(y^{\prime})-1\otimes [x,y]_L\otimes a[x^{\prime},y^{\prime}]_L$;
\end{enumerate}
where $x,x^{\prime},y,y^{\prime},z\in L,~a\in A$. Let us denote the quotient $A$-module 
$$\mathfrak{uce}^{\phi}_A L:=A\otimes L\otimes L/M^{\phi}_A L$$
and we denote any coset $a\otimes x\otimes y +M^{\phi}_A L$ simply by $(a,x,y)$. From the definition of $\mathfrak{uce}^{\phi}_A L$ the following identities hold. 
\begin{enumerate}
\item $(a,x,y)=-(a,y,x)$;
\item $(a,\alpha_L(x),[y,z]_L)+(a,\alpha_L(y),[z,x]_L)+(a,\alpha_L(z),[x,y]_L)=0$;
\item $(1,[x,y]_L,a[x^{\prime},y^{\prime}]_L)= (\phi(a),[x,y]_L,[x^{\prime},y^{\prime}]_L)+(\rho([x,y]_L)(a),\alpha_L(x^{\prime}),\alpha_L(y^{\prime}))$,
\end{enumerate}
for any $x,x^{\prime},y,y^{\prime},z\in L,~a\in A$. Define an $A$-module homomorphism 
$$\Psi: A\otimes L\otimes L\rightarrow L$$
given by $\Psi(a,x,y)=a[x,y]_L$. Since the map $\Psi$ vanishes on the submodule $M^{\phi}_A L$ it induces an $A$-linear map $u_L:\mathfrak{uce}^{\phi}_A L\rightarrow L$.

The tuple $(A,\mathfrak{uce}^{\phi}_A L,[-,-],\phi,\tilde{\alpha_L},\tilde{\rho_L})$ is a hom-Lie-Rinehart algebra over $(A,\phi)$.
\begin{itemize}
\item The bracket $[-,-]$ is defined as
\begin{equation*}
\begin{split}
&[(a_1,x_1,y_1),(a_2,x_2,y_2)]\\
= &(\phi(a_1a_2),[x_1,y_1]_L,[x_2,y_2]_L)
+(\phi(a_1).[x_1,y_1]_L(a_2),\alpha_L(x_2),\alpha_L(y_2))\\
&-(\phi(a_2).[x_2,y_2]_L(a_1),\alpha_L(x_1),\alpha_L(y_1));
\end{split}
\end{equation*}
\item The endomorphism $\tilde{\alpha_L}:\mathfrak{uce}^{\phi}_A L \rightarrow \mathfrak{uce}^{\phi}_A L$ is defined as 
$$\tilde{\alpha_L}(a,x,y)=(\phi(a),\alpha_L(x),\alpha_L(y));$$
\item The anchor $\tilde{\rho_L}:\mathfrak{uce}^{\phi}_A L\rightarrow Der_{\phi}A$ is defined as 
$$\tilde{\rho_L}(a,x,y)(b)=\phi(a)\rho_L([x,y]_L)(b),$$
\end{itemize} 
for any $x_1,x_2,x,y_1,y_2,y\in L,~a,b,a_1,a_2\in A$. Let us denote 
$$\mathfrak{uce}^{\phi}_A(\mathcal{L},\alpha_L):=(A,\mathfrak{uce}^{\phi}_A L,[-,-],\phi,\tilde{\alpha_L},\tilde{\rho_L}).$$
Then the previously defined map $u_L: \mathfrak{uce}^{\phi}_A(\mathcal{L},\alpha_L)\rightarrow (\mathcal{L},\alpha_{L})$ is a morphism in the category $hLR_A^{\phi}$. 
 
Let $f:(\mathcal{L},\alpha_L)\rightarrow (\mathcal{M},\alpha_M)$ be an arbitrary morphism in the category $hLR^A_{\phi}$. Define a map
 $\mathfrak{uce}_A^{\phi}(f):\mathfrak{uce}_A^{\phi}(L)\rightarrow \mathfrak{uce}_A^{\phi}(M)$ as follows:
 $$\mathfrak{uce}_A^{\phi}(f)(a,x,y)=(a,f(x),f(y))$$
for $a\in A$, and $x,y\in L$. Then from the definition we get that 
the map $\mathfrak{uce}_A^{\phi}(f)$
is a morphism in the category $hLR_A^{\phi}$ and we have the following commutative diagram:
$$\begin{CD}
\mathfrak{uce}^{\phi}_A(\mathcal{L},{\alpha_L})@>\mathfrak{uce}_A^{\phi}(f)>> \mathfrak{uce}_A^{\phi}(\mathcal{M},\alpha_M)\\
  @V u_L VV  @V u_M VV \\
(\mathcal{L},\alpha_L)@>f>> (\mathcal{M},\alpha_M)
\end{CD}$$ 
 This in turn implies that $\mathfrak{uce}_A^{\phi}:hLR_A^{\phi}\rightarrow hLR_A^{\phi}$ is a functor. 
 
Let $\{L,L\}$ denotes the $A$-submodule of $L$ generated by elements of the form: $[x,y]_L$ for $x,y\in L$. Then we get a hom-Lie-Rinehart subalgebra $(\{\mathcal{L},\mathcal{L}\},\alpha_{\{L,L\}})$ of $(\mathcal{L},\alpha_L)$ by restricting the hom-Lie bracket, the anchor map, and the endomorphism $\alpha_L$ on the $A$-submodule $\{L,L\}$ (here $\alpha_{\{L,L\}}$ denotes the restriction of $\alpha_L$ on $\{L,L\}$.). Then, it easily follows that $u_L : \mathfrak{uce}^{\phi}_A(\mathcal{L},\alpha_L)\rightarrow (\{\mathcal{L},\mathcal{L}\},\alpha_{[L,L]})$ is a surjective homomorphism and it gives a central extension of the hom-Lie-Rinehart algebra $(\{\mathcal{L},\mathcal{L}\},\alpha_{[L,L]})$. 

\begin{Thm}\label{Existence 1}
Let $(\mathcal{L},\alpha_L)$ be a perfect hom-Lie-Rinehart algebra. Then the short exact sequence:
$$(\mathcal{P},\alpha_P)\rightarrow \mathfrak{uce}^{\phi}_A(\mathcal{L},{\alpha_L})\xrightarrow{u_L} (\mathcal{L},\alpha_L)$$
is a universal central extension of $(\mathcal{L},\alpha_L)$, where the underlying $A$-module in $(\mathcal{P},\alpha_P)$ is $P=Ker(u_L)$, and the map $\alpha_P$ is restriction of $\tilde{\alpha_L}$ on $Ker(u)$.

\begin{proof}
Let $(\mathcal{N},\alpha_N)\xrightarrow{i} (\mathcal{M}^{\prime},\alpha_{M^{\prime}}) \xrightarrow{\sigma}(\mathcal{L},\alpha_L)$ be a central extension of $(L,\alpha_L)$. Let $s:L\rightarrow M^{\prime}$ be a map such that $ \sigma \circ s= Id_L$
, then we have the following observations.
\begin{enumerate}
\item $s(km+n)-ks(m)-s(n)\in Ker(\sigma)\subset Z_A(\mathcal{M}^{\prime}),$
\item $s[m,n]_L-[s(m),s(n)]_{M^{\prime}}\in Ker(\sigma)\subset Z_A(\mathcal{M}^{\prime}),$
\item $s(a.m)-a.s(m)\in Ker(\sigma)\subset Z_A(\mathcal{M}^{\prime}),$
\item $s(\alpha_L(m))-\alpha_M^{\prime}(s(m))\in Ker(\sigma)\subset Z_A(\mathcal{M}^{\prime}),$
\item $s(m)(a)=\sigma(s(m))(a)=m(a),$
\end{enumerate}
for any $m,n\in L,~a\in A,$ and $k\in R$. Let us define a map $\psi:A\times L\times L\rightarrow M^{\prime}$ by
 $$\psi(a,x,y)=a[s(x),s(y)]_{M^{\prime}}$$
for all $x,y\in L$. By the above observations the map $\psi$ extends to a hom-Lie-Rinehart algebra homomorphism
 $$\tau:\mathfrak{uce}^{\phi}_A(\mathcal{ L},{\alpha_L})\rightarrow(\mathcal{M}^{\prime},\alpha_{M^{\prime}}).$$
Next, by using the Lemma \ref{uniq} we get that the map $\tau$ is unique on the subalgebra $\{\mathfrak{uce}^{\phi}_A L,\mathfrak{uce}^{\phi}_A L\}$. Since $(\mathcal{L},\alpha_L)$ is a perfect hom-Lie-Rinehart algebra, the hom-Lie-Rinehart algebra $\mathfrak{uce}^{\phi}_A(\mathcal{L},\alpha_L)$ is also perfect, i.e. $\{\mathfrak{uce}^{\phi}_A L,\mathfrak{uce}^{\phi}_A L\}=\mathfrak{uce}^{\phi}_A L$. Therefore, there exists a unique homomorphism $\tau:\mathfrak{uce}^{\phi}_A(\mathcal{L},\alpha_L)\rightarrow(\mathcal{M}^{\prime},\alpha_{M^{\prime}})$ such that $\sigma\circ \tau=u_L$. This completes the proof.
\end{proof}

\end{Thm}

\section{Universal $\alpha$-central extensions of hom-Lie-Rinehart algebras}
In this section, we assume that $A$ is an associative commutative unital $R$-algebra and $\phi:A\rightarrow A$ is an algebra epimorphism. We define $\alpha$-perfect hom-Lie-Rinehart algebras and prove that for an $\alpha$-perfect hom-Lie-Rinehart algebra, the functor $\mathfrak{uce}_A^{\phi}$ associates a universal $\alpha$-central extension. 
\begin{Def}
An extension of a hom-Lie-Rinehart algebra $(\mathcal{L},\alpha_L)$
$$\begin{CD}
(\mathcal{M},\alpha_M)@>i>> (\mathcal{L}^{\prime},\alpha_{L^{\prime}}) @> \sigma>>(\mathcal{L},\alpha_L)
\end{CD}$$
is called {\bf  $\alpha$-central extension} of $(\mathcal{L},\alpha_L)$ if $i(\alpha_M(M))\subset Z_A(\mathcal{L}^{\prime},\alpha_{L^{\prime}})$.  
\end{Def}

\begin{Rem}\label{Comp1}
Any central extension of a hom-Lie-Rinehart algebra is an $\alpha$-central extension, but the converse may not be true. In the classical cases of Lie algebras, Leibniz algebras, Lie-Rinehart algebras and other classical algebras, the composition of two central extensions is again a central extension. This property does not hold for hom-Lie-Rinehart algebras. In fact, it does not even hold in the case of hom-Lie algebras, see Example 4.9 in \cite{CIP15}. In the following Lemma, we prove that composition of two central extensions in $hLR_A^{\phi}$, is an $\alpha$-central extension. 
\end{Rem}
\begin{Lem}\label{Comp}
 Suppose $(\mathcal{M},\alpha_M)\xrightarrow{i} (\mathcal{K},\alpha_K) \xrightarrow{\sigma}(\mathcal{L},\alpha_L)$
is a central extension where $(\mathcal{K},\alpha_K)$ is a perfect hom-Lie-Rinehart algebra.Then for a central extension of $(\mathcal{K},\alpha_K)$ given by $
(\mathcal{N},\alpha_N)\xrightarrow{j} (\mathcal{L}^{\prime},\alpha_{L^{\prime}}) \xrightarrow{\tau}(\mathcal{K},\alpha_K) $  the composition 
$$(\mathcal{N}^{\prime},\alpha_{N^{\prime}})\xrightarrow{j^{\prime}} (\mathcal{L}^{\prime},\alpha_{L^{\prime}}) \xrightarrow{\sigma\circ\tau}(\mathcal{L},\alpha_L)
$$ is an $\alpha$-central extension.
\end{Lem}

\begin{proof}
Here, the hom-Lie-Rinehart algebra $(\mathcal{K},\alpha_K)$ is perfect. So, any element $X\in K$ can be written as follws:
 $$X=\sum_i a_i[X_{i_1},X_{i_2}]$$ for some $a_i\in A$ and $X_{i_1},X_{i_2}\in K$. Assume $Y\in {L}^{\prime}$, then $\tau(Y)=\sum_i a_i[k_{i_1},k_{i_2}]$ for some $k_{i_1},k_{i_2}\in K$ and $a_i\in A$. Since $\tau$ is surjective, we have $\tau(Y)=\sum_i a_i\tau[Y_{i_1},Y_{i_2}]$. So, $Y-\sum_i a_i[Y_{i_1},Y_{i_2}]\in Ker(\tau)$. So, we can write any $Y\in L^{\prime}$ as \begin{equation}\label{Y}
Y=\sum_i a_i[Y_{i_1},Y_{i_2}]+\eta,
\end{equation}
for some $\eta\in Ker(\tau)$. Now, using the expression \eqref{Y} for $Y\in L^{\prime}$, $Ker(\tau)\subset Z_A(\mathcal{L}^{\prime},\alpha_{L^{\prime}})$ and the hom-Jacobi identity, we have
$$[a.\alpha_{N^{\prime}}(n),Y]=[\alpha_{N^{\prime}}(a^{\prime}.n),Y]=0$$
for each $n\in N^{\prime}=Ker(\sigma\circ\tau)$, and $a\in A$ (since $\phi$ is surjective there exists $a^{\prime}\in A$ such that $\phi(a^{\prime})=a$). Moreover, for each $n\in N^{\prime }$, and $a\in A$, we have $n(a)=0$. Therefore, $\alpha_{N^{\prime}}(Ker(\sigma\circ \tau))\subset Z_A(\mathcal{L}^{\prime})$, i.e
the composition  
$$(\mathcal{N}^{\prime},\alpha_{N^{\prime}})\xrightarrow{j^{\prime}} (\mathcal{L}^{\prime},\alpha_{L^{\prime}}) \xrightarrow{\sigma\circ\tau}(\mathcal{L},\alpha_L)$$ 
is an $\alpha$-central extension.

\end{proof}

\begin{Def}
A hom-Lie-Rinehart algebra $(\mathcal{L},\alpha_L)$ over $(A,\phi)$ is called {\bf $\alpha$-perfect} if $L=\{\alpha_L(L),\alpha_L(L)\}$.   
\end{Def}

Note that any $\alpha$-perfect hom-Lie-Rinehart algebra $(\mathcal{L},\alpha_L)$ is also perfect. Moreover, since $\phi: A\rightarrow A$ is an algebra epimorphism, the map $\alpha_L$ is surjective. 

In particular, if $A=R$, and the map $\phi=Id_A$, then any $\alpha$-perfect hom-Lie-Rinehart algebra $(\mathcal{L},\alpha_L)$ is just an $\alpha$-perfect hom-Lie algebra. 
\begin{Lem}\label{uniq2}
Let $(\mathcal{K},\alpha_K)$ be an $\alpha$-perfect hom-Lie-Rinehart algebra and the short exact sequence:
$$\begin{CD}
(\mathcal{M}^{\prime},\alpha_{M^{\prime}})@>j>> (\mathcal{L}^{\prime},\alpha_{L^{\prime}}) @> \tau>>(\mathcal{L},\alpha_L)
\end{CD}$$ 
be an $\alpha$-central extension of $(\mathcal{L},\alpha_L)$. If there exist hom-Lie-Rinehart algebra homomorphisms $f,g:(\mathcal{K},\alpha_K)\rightarrow (\mathcal{L}^{\prime},\alpha_{L^{\prime}})$ such that $\tau\circ f=\tau\circ g$, then $f=g$.
\end{Lem}
\begin{proof}
The proof is similar to the proof of Lemma \ref{uniq}
\end{proof}
\begin{Def}
A central extension 
$$\begin{CD}
(\mathcal{M},\alpha_M)@>i>> (\mathcal{K},\alpha_K) @> \sigma>>(\mathcal{L},\alpha_L)
\end{CD}$$
is called a {\bf universal $\alpha$-central extension} if for every $\alpha$-central extension 
$$\begin{CD}
(\mathcal{M}^{\prime},\alpha_{M^{\prime}})@>j>> (\mathcal{L}^{\prime},\alpha_{L^{\prime}}) @> \tau>>(\mathcal{L},\alpha_L)
\end{CD}$$
there exists a unique homomorphism $h:(\mathcal{K},\alpha_K)\rightarrow (\mathcal{L}^{\prime},\alpha^{\prime}) $, in the category $hLR_A^{\phi}$  such that $\tau\circ h=\sigma$.
\end{Def}
Any central extension of a hom-Lie-Rinehart algebra is an $\alpha$-central extension.
So, every universal $\alpha$-central extension of a hom-Lie-Rinehart algebra is a universal central extension.

\begin{Thm}\label{splitting}
Let $(\mathcal{M},\alpha_M)\xrightarrow{i} (\mathcal{K},\alpha_K) \xrightarrow{\sigma}(\mathcal{L},\alpha_L)$ be a central extension  of a hom-Lie-Rinehart algebra $(\mathcal{L},\alpha_L)$. If the central extension is a universal $\alpha$-central extension, then the hom-Lie-Rinehart algebra $(\mathcal{K},\alpha_K)$ is perfect and every central extension of $(\mathcal{K},\alpha_K)$ splits uniquely in the category $hLR_A^{\phi}$. 
\end{Thm}

\begin{proof}
Let us assume that $(\mathcal{K},\alpha_K)$ is not perfect. Now consider the product $K^{\prime}=K\times K/\{K,K\}$ as an $A$-module. Then, $(A,K^{\prime},[-,-]_{K^{\prime}},\phi,\alpha_{K^{\prime}},\rho_{K^{\prime}})$ is a hom-Lie-Rinehart algebra, where 
\begin{enumerate}
\item $[(x,y+\{K,K\}),(x^{\prime},y^{\prime}+\{K,K\})]_K^{\prime}=([x,x^{\prime}],0);$ 
\item $\alpha_{K^{\prime}}(x,y+\{K,K\})=(\alpha_K(x),\alpha_K(y)+\{K,K\});$
\item $\rho_{K^{\prime}}(x,y+\{K,K\})(a)=\rho_K(x)(a).$
\end{enumerate} 

Let us consider the central extension $(\mathcal{M^{\prime}},\alpha_{M^{\prime}})\xrightarrow{i} (\mathcal{K^{\prime}},\alpha_{K^{\prime}}) \xrightarrow{\tilde{\sigma}}(\mathcal{L},\alpha_L)$ where $K^{\prime}=K\times K/\{K,K\}$, $\tilde{\sigma}:K^{\prime}\rightarrow M$ is given by $\tilde{\sigma}(x,y+\{K,K\})=\sigma(x)$ for $x,y\in K$ and $M^{\prime}=Ker(\tilde{\alpha})$. Now let us define the following maps:
\begin{enumerate}
\item $f:K\rightarrow K\times \{K,K\}$ given by $f(x)=(x,0)$ for all $x\in K$;
\item $g:K\rightarrow K\times \{K,K\}$ given by $g(x)=(x,x+\{K,K\})$ for all $x\in K$.
\end{enumerate}
Then $\tilde{\sigma}\circ f=\tilde{\sigma}\circ g=\sigma$. Since $(\mathcal{M},\alpha_M)\xrightarrow{i} (\mathcal{K},\alpha_K) \xrightarrow{\sigma}(\mathcal{L},\alpha_L)$ is a universal $\alpha$-central extension, it is also a universal central extension. Therefore, by universal property $f=g$, i.e. $K/\{K,K\}=0$. So, $(\mathcal{K},\alpha_K)$ is a perfect hom-Lie-Rinehart algebra.

Let $(\mathcal{N},\alpha_N)\xrightarrow{j} (\mathcal{L^{\prime}},\alpha_{L^{\prime}}) \xrightarrow{\tau}(\mathcal{K},\alpha_K)$  be a central extension of $(\mathcal{K},\alpha_K)$. By Lemma \ref{Comp}, $(\mathcal{P},\alpha_P)\xrightarrow{k} (\mathcal{L^{\prime}},\alpha_{L^{\prime}}) \xrightarrow{\sigma\circ \tau}(\mathcal{L},\alpha_L)$ is an $\alpha$-central extension, where $P=Ker(\sigma\circ\tau)$. Since $(\mathcal{M},\alpha_M)\xrightarrow{i} (\mathcal{K},\alpha_K) \xrightarrow{\sigma}(\mathcal{L},\alpha_L)$ is a universal $\alpha$-central extension, we have a unique homomorphism $\gamma:(\mathcal{K},\alpha_K)\rightarrow (\mathcal{L^{\prime}},\alpha_{L^{\prime}})$ such that $\sigma\circ(\tau\circ\gamma)=\sigma $. Since $(\mathcal{K},\alpha_K)$ is perfect, by using Lemma \ref{uniq}, we get $\tau\circ\gamma=Id_K$. Thus, every central extension of $(\mathcal{K},\alpha_K)$ splits uniquely in $hLR_A^{\phi}$.
\end{proof}

Now, we prove that the functor $\mathfrak{uce}_A^{\phi}:hLR_A^{\phi}\rightarrow hLR_A^{\phi}$ associates a universal $\alpha$-central extension to an $\alpha-$perfect hom-Lie-Rinehart algebra. In particular we have the following theorem:

\begin{Thm}
Let $(\mathcal{L},\alpha_L)$ be an $\alpha$-perfect hom-Lie-Rinehart algebra. Then 
$$ (\mathcal{P},\alpha_P)\rightarrow \mathfrak{uce}^{\phi}_A (\mathcal{L}, \alpha_L)\xrightarrow{u_L} (\mathcal{L},\alpha_L)$$
is a universal $\alpha$-central extension of the hom-Lie-Rinehart algebra $(\mathcal{L},\alpha_L)$.

\begin{proof}
Note that $(\mathfrak{uce}^{\phi}_A L,\tilde{\alpha_L})$ is also an $\alpha$-perfect hom-Lie-Rinehart algebra. Let $$ 
(\mathcal{N},\alpha_N)\xrightarrow{i} (\mathcal{M}^{\prime},\alpha_{M^{\prime}}) \xrightarrow{\sigma}(\mathcal{L},\alpha_L)  
$$ be an $\alpha$-central extension of $(L,\alpha_L)$, i.e. $\alpha_{M^{\prime}}(ker(\sigma))\subset Z_A(\mathcal{M}^{\prime},\alpha_{M^{\prime}})$. Let $s:L\rightarrow M^{\prime}$ be a map such that $\sigma \circ s=Id_L$, then we
define a map $\psi:A\times L\times L\rightarrow M^{\prime}$ given by
 $$\psi(a,\alpha_L(x),\alpha_L(y))=a[\alpha_{M^{\prime}}(s(x)),\alpha_{M^{\prime}}(s(y))]_{M^{\prime}}$$
for all $x,y\in L$. Following the arguments as given in  the proof of Theorem \ref{Existence 1}, it is clear that $\psi$ extends to a hom-Lie-Rinehart algebra homomorphism
 $$\tau:\mathfrak{uce}^{\phi}_A(\mathcal{ L},\alpha_L)\rightarrow(\mathcal{M}^{\prime},\alpha_{M^{\prime}}).$$
Next, by using Lemma \ref{uniq2}, the map $\tilde{f}$ is a unique homomorphism such that $\sigma\circ \tau=u_L$.
\end{proof}
\end{Thm}
\begin{Rem}
In the case of hom-Lie algebras, universal $\alpha$-central extensions are closely related to the homology of the hom-Lie algebra with trivial coefficients (see \cite{CIP15, CKP17}). However for a hom-Lie-Rinehart algebra $(\mathcal{L},\alpha_L)$ over $(A,\phi)$, such relations are not possible since there does not exist a canonical right $(\mathcal{L},\alpha_L)$-module structure on the pair $(A,\phi)$. Such relations are not even possible in the case of Lie-Rinehart algebras \cite{CGML14}.  
\end{Rem}
\begin{Def}\label{Universality1}
An $\alpha$-central extension 
$(\mathcal{M},\alpha)\xrightarrow{i} (\mathcal{K},\alpha_K) \xrightarrow{\sigma}(\mathcal{L},\alpha_L)$  of a hom-Lie-Rinehart algebra $(\mathcal{L},\alpha_L)$ is said to be {\bf universal} if for any other central extension
$
(\mathcal{M}^{\prime},\alpha_{M^{\prime}})\xrightarrow{j} (\mathcal{L}^{\prime},\alpha_{L^{\prime}}) \xrightarrow{\tau}(\mathcal{L},\alpha_L)$  of $(\mathcal{L},\alpha_L)$,
there exists a unique morphism $h:(\mathcal{K},\alpha_K)\rightarrow (\mathcal{L}^{\prime},\alpha_{L^{\prime}})$ in the category $hLR_A^{\phi}$ such that $\tau\circ h=\sigma$.
\end{Def}

\begin{Prop}\label{Compwithuniv}
Let $(\mathcal{M},\alpha_M)\xrightarrow{i} (\mathcal{K},\alpha_K) \xrightarrow{\sigma}(\mathcal{L},\alpha_L)$ be a central extension of $(\mathcal{L},\alpha_L)$, and the extension $
(\mathcal{N},\alpha_N)\xrightarrow{j} (\mathcal{L}^{\prime},\alpha_{L^{\prime}}) \xrightarrow{\tau}(\mathcal{K},\alpha_K)$ be a universal $\alpha$-central extension, then the $\alpha$-central extension 
$$(\mathcal{N}^{\prime},\alpha_{N^{\prime}})\xrightarrow{j^{\prime}} (\mathcal{L}^{\prime},\alpha_{L^{\prime}}) \xrightarrow{\sigma\circ\tau}(\mathcal{L},\alpha_L)
$$ is universal in the sense of Definition \ref{Universality1}.
\end{Prop}
\begin{proof}
Since $
(\mathcal{N},\alpha_N)\xrightarrow{j} (\mathcal{L}^{\prime},\alpha_{L^{\prime}}) \xrightarrow{\tau}(\mathcal{K},\alpha_K) $, is a universal $\alpha$-central extension, by Theorem \ref{splitting} the hom-Lie-Rinehart algebra $(\mathcal{L}^{\prime},\alpha_{L^{\prime}})$ is perfect and consequently $(\mathcal{K},\alpha_K)$ is also perfect. Next, by Lemma \ref{Comp}, the composition $(\mathcal{N}^{\prime},\alpha_{N^{\prime}})\xrightarrow{j^{\prime}} (\mathcal{L}^{\prime},\alpha_{L^{\prime}}) \xrightarrow{\sigma\circ\tau}(\mathcal{L},\alpha_L)
$ is an $\alpha-$central extension. Let 
$$\begin{CD}
(\mathcal{H},\alpha_H)@>k>> (\mathcal{F},\alpha_F) @> \gamma>>(\mathcal{L},\alpha_L)
\end{CD}$$
be a central extension. Then similar to the proof of Theorem \ref{Characterisation}, we can construct a central extension  of $(\mathcal{L}^{\prime},\alpha_{L^{\prime}})$ via the pull-back diagram with respect to homomorphisms $\sigma\circ \tau$ and $\gamma$ of $(\mathcal{L}^{\prime},\alpha_{L^{\prime}})$) given by $$(\mathcal{N^{\prime\prime}},\alpha_{N^{\prime\prime}})\xrightarrow{k^{\prime}} (\mathcal{P},\alpha_{P}) \xrightarrow{p_{L^{\prime}}}(\mathcal{L}^{\prime},\alpha_{L^{\prime}}),$$
where the $A$-module $P$ is given as follows:
 $$P=\{(l,f)\in L^{\prime}\times_{Der_{\phi}(A)} F: \sigma\circ\tau(l)=\gamma(f)\}$$
The map $\alpha_P:P\rightarrow P$ is defined by: $\alpha_P(l,f)=(\alpha_{L^{\prime}}(l),\alpha_F(f))$. Note that every central extension of $(\mathcal{L}^{\prime},\alpha_{L^{\prime}})$ splits uniquely by Theorem \ref{splitting}. Therefore the extension $(\mathcal{N^{\prime\prime}},\alpha_{N^{\prime\prime}})\xrightarrow{k^{\prime}} (\mathcal{P},\alpha_{P}) \xrightarrow{p_{L^{\prime}}}(\mathcal{L}^{\prime},\alpha_{L^{\prime}})$ splits uniquely and the splitting gives a unique homomorphism $h:(\mathcal{L}^{\prime},\alpha_{L^{\prime}})\rightarrow (\mathcal{F},\alpha_F)$ such that $\gamma\circ h= \sigma\circ \tau$. Hence, the $\alpha$-central extension $(\mathcal{N}^{\prime},\alpha_{N^{\prime}})\xrightarrow{j^{\prime}} (\mathcal{L}^{\prime},\alpha_{L^{\prime}}) \xrightarrow{\sigma\circ\tau}(\mathcal{L},\alpha_L)
$ is universal in the sense of Definition \ref{Universality1}.
\end{proof}

\section{Lifting of Automorphisms and Derivations to Central Extensions}
Let $
(\mathcal{N},\alpha_N)\xrightarrow{j} (\mathcal{K} ,\alpha_{K}) \xrightarrow{f}(\mathcal{L},\alpha_L) $
be a central extension, where $(\mathcal{K} ,\alpha_{K})$ is an $\alpha$-perfect hom-Lie-Rinehart algebra. Then $(\mathcal{L},\alpha_L)$ is also an $\alpha$-perfect hom-Lie-Rinehart algebra and we get the following commutative diagram:
$$\begin{CD}\label{diagram}
\mathfrak{uce}^{\phi}_A(\mathcal{K},\alpha_{K})@>uce^{\phi}_A(f)>> \mathfrak{uce}^{\phi}_A(\mathcal{L},\alpha_{L})\\
  @V u_K VV  @V u_L VV \\
(\mathcal{K},\alpha_{K})@>f>> (\mathcal{L},\alpha_{L})
\end{CD}$$ 
Since the central extension $(\mathcal{N}_K,\alpha_{N_K})\xrightarrow{i}\mathfrak{uce}^{\phi}_A(\mathcal{K},\alpha_{K})\xrightarrow{u_K} (\mathcal{K} ,\alpha_{K})$ is a universal $\alpha$-central extension, by using Proposition \ref{Compwithuniv}, the $\alpha-$central extension induced by the map $\mathfrak{uce}^{\phi}_A(\mathcal{K},\alpha_{K})\xrightarrow{f\circ u_K} (\mathcal{L},\alpha_L)$
is universal in the sense of Definition \ref{Universality1}. Consequently, we get a unique homomorphism $$h:\mathfrak{uce}^{\phi}_A(\mathcal{K},\alpha_{K})\rightarrow \mathfrak{uce}^{\phi}_A(\mathcal{L},\alpha_{L})$$ such that $u_L\circ h=f\circ u_K$, i.e. $h=uce^{\phi}_A(f)$. 

Next, since $(\mathcal{N}_L,\alpha_{N_L})\xrightarrow{i}\mathfrak{uce}^{\phi}_A(\mathcal{L},\alpha_{L})\xrightarrow{u_L} (\mathcal{L},\alpha_{L})$ is a universal $\alpha$-central extension, we have a unique homomorphism $g:\mathfrak{uce}^{\phi}_A(\mathcal{L},\alpha_{L})\rightarrow \mathfrak{uce}^{\phi}_A(\mathcal{K},\alpha_{K})$ such that $f\circ u_K\circ g=u_L$. Now, by using Lemma \ref{uniq2}, it follows that $uce^{\phi}_A(f)\circ g=Id_{\mathfrak{uce}^{\phi}_A(\mathcal{L},\alpha_{L})}$ and $g\circ uce^{\phi}_A(f)=Id_{\mathfrak{uce}^{\phi}_A(\mathcal{K},\alpha_{K})}$, i.e. we have an isomorphism:
$$\mathfrak{uce}^{\phi}_A(\mathcal{K},\alpha_{K})\cong \mathfrak{uce}^{\phi}_A(\mathcal{L},\alpha_{L}).$$

It is immediate to see that we get a central extension: $$(\mathcal{P},\alpha_P)\xrightarrow{k}\mathfrak{uce}^{\phi}_A(\mathcal{L},\alpha_{L})\xrightarrow{u_K\circ \mathfrak{uce}_A^{\phi}(f)^{-1}}(\mathcal{K},\alpha_K)$$
where the underlying $A$-module $P:=Ker(u_K\circ \mathfrak{uce}_A^{\phi}(f)^{-1})=\mathfrak{uce}_A^{\phi}(f)(Ker(u_K))$ in the ideal $(\mathcal{P},\alpha_P)$ of hom-Lie-Rinehart algebra $\mathfrak{uce}^{\phi}_A(\mathcal{L},\alpha_{L})$.

Let us denote the set of automorphisms of a hom-Lie-Rinehart algebra $(\mathcal{L},\alpha_L)$ in the category $hLR_A^{\phi}$ by $Aut(\mathcal{L},\alpha_L)$. Then we have the following result for the lifting of automorphisms.

\begin{Thm}\label{Lift1}
Let $
(\mathcal{N},\alpha_N)\xrightarrow{j} (\mathcal{K} ,\alpha_{K}) \xrightarrow{f}(\mathcal{L},\alpha_L) $
be a central extension, where $(\mathcal{K} ,\alpha_{K})$ is an $\alpha$-perfect hom-Lie-Rinehart algebra. If $h\in Aut(\mathcal{L},\alpha_L)$, then there exists a unique automorphism $\tilde{h} \in Aut(\mathcal{K},\alpha_K)$ such that 
$h\circ f=f\circ \tilde{h}$ (i.e. there exists a lifting of $h$) if and only if $\mathfrak{uce}_A^{\phi}(h)(P)=P$. Moreover, $\tilde{h}(Ker(f))=Ker(f)$ and we have a group isomorphism: 
$$\Phi:\{h\in Aut(\mathcal{L},\alpha_L):\mathfrak{uce}_A^{\phi}(h)(P)=P\}\rightarrow \{g\in Aut(\mathcal{K},\alpha_K):g(Ker(f))=Ker(f)\}$$
given by $\Phi(h)=\tilde{h}$.
\end{Thm}

\begin{proof}
Let us assume there exists $\tilde{h} \in Aut(\mathcal{K},\alpha_K)$ such that 
$h\circ f=f\circ \tilde{h}$, then first note that
$$\mathfrak{uce}_A^{\phi}(h)\circ \mathfrak{uce}_A^{\phi}(f)=\mathfrak{uce}_A^{\phi}(f)\circ \mathfrak{uce}_A^{\phi}(\tilde{h}).$$
Then, $$\mathfrak{uce}_A^{\phi}(h)(\mathfrak{uce}_A^{\phi}(f)(Ker(u_K)))=(\mathfrak{uce}_A^{\phi}(f)\circ \mathfrak{uce}_A^{\phi}(\tilde{h}))(Ker(u_K))=\mathfrak{uce}_A^{\phi}(f)(Ker(u_K)),$$
i.e., $\mathfrak{uce}_A^{\phi}(h)(P)=P$. 

Conversely, Let us assume that $\mathfrak{uce}_A^{\phi}(h)(P)=P$. Next, consider the central extension: $$(\mathcal{P},\alpha_P)\xrightarrow{k}\mathfrak{uce}^{\phi}_A(\mathcal{L},\alpha_{L})\xrightarrow{u_K\circ \mathfrak{uce}_A^{\phi}(f)^{-1}}(\mathcal{K},\alpha_K).$$
Define a map $\tilde{h}:(\mathcal{K},\alpha_K)\rightarrow (\mathcal{K},\alpha_K)$ as follows: 
$$\tilde{h}(k)=u_K\circ \mathfrak{uce}_A^{\phi}(f)^{-1}(\mathfrak{uce}_A^{\phi}(h)(\bar{k}))$$
for any $k\in K$. Here $\bar{k}\in \mathfrak{uce}_A^{\phi}L$ such that $u_K\circ \mathfrak{uce}_A^{\phi}(f)^{-1}(\bar{k})=k$. Since $\mathfrak{uce}_A^{\phi}(h)(P)=P$, it is immediate to see that $Ker(u_K\circ \mathfrak{uce}_A^{\phi}(f)^{-1}\circ \mathfrak{uce}_A^{\phi}(h) )=Ker(u_K\circ \mathfrak{uce}_A^{\phi}(f)^{-1})=P$. Then, it easily follows by the definition of the map $\tilde{h}$ that it is an automorphism. By using the condition: $u_L=f\circ u_K\circ \mathfrak{uce}_A^{\phi}(f)^{-1}$, we get $f\circ\tilde{h}=h\circ f$. The uniqueness of the map $\tilde{h}$ follows from Lemma \ref{uniq}. Since, we have $f\circ\tilde{h}=h\circ f$, the condition: $\tilde{h}(Ker(f))=Ker(f)$ also follows. 
 
 If $h\in Aut(\mathcal{L},\alpha_L)$, and $\mathfrak{uce}_A^{\phi}(h)(P)=P$, then it lifts to a unique automorphism $\tilde{h} \in Aut(\mathcal{K},\alpha_K)$ such that 
$h\circ f=f\circ \tilde{h}$. Thus the map $\Phi$ is well defined. Next, using Lemma \ref{uniq}, the map $\Phi$ is also injective. Let $g\in Aut(\mathcal{K},\alpha_K)$ such that $g(Ker(f))=Ker(f)$, then define a map $h:(\mathcal{L},\alpha_L)\rightarrow (\mathcal{L},\alpha_L)$ as follows:
$$h(x)=f(g(\bar{x}))$$
for any $x\in L$, here $\bar{x}\in K$ such that $f(\bar{x})=x$. Using the property $g(Ker(f))=Ker(f)$ it follows that $h\in Aut(\mathcal{L},\alpha_L)$. It is clear that $g\in Aut(\mathcal{K},\alpha_K)$ is a lifting of $h\in Aut(\mathcal{L},\alpha_L)$ and since lifting exists, we have $\mathfrak{uce}_A^{\phi}(h)(P)=P$. Hence, $\Phi$ is a group isomorphism.   

\end{proof}
\begin{Def}\label{alpha-derivations}
Let $(\mathcal{L},\alpha_L)$ be a hom-Lie-Rinehart algebra over $(A,\phi)$. A map $D:L\rightarrow L$ is an {\bf $\alpha$-derivation} of the hom-Lie-Rinehart algebra, if 
\begin{itemize}
\item $D\circ \alpha_L=\alpha_L\circ D;$
\item $D[x,y]=[Dx,\alpha_L(y)]+[\alpha_L(x),Dy]$ for any $x,y\in L$;
\item there exists a $\phi$-derivation $\sigma_D\in Der_{\phi}(A)$, called symbol of $D$, such that\\ $\sigma_D\circ \phi=\phi\circ \sigma_D$, and $D(a.x)=\phi(a).D(x)+\sigma_D(a).\alpha(x)$ for any $x\in L$, and $a\in A$;
\item $\sigma_D(x(a))=\alpha(x)(\sigma_D(a))+D(x)(\phi(a))$ for any for $x\in L$, and $a\in A$. 
\end{itemize}  
\end{Def}

When $A=R$, and $\phi=Id_R$, $(\mathcal{L},\alpha_L)$ becomes a hom-Lie algebra and an $\alpha$-derivation of $(\mathcal{L},\alpha_L)$ is an $\alpha$-derivation of hom-Lie algebra $(L,[-,-]_L,\alpha_L)$ defined in \cite{Sheng12}. 
\begin{Exm}
Let $(\mathcal{L},\alpha_L)$ be a hom-Lie-Rinehart algebra over $(A,\phi)$ and $x\in L$ such that $\alpha_L(x)=x$. Then define a map $D:L\rightarrow L$ as: $D(y)=[\alpha_L(x),y]$ for any $y\in L$. Consider the $\phi-$derivation $\sigma_D\in Der_{\phi}(A)$ defined as: $\sigma_D(a)=x(a)$. Then $D$ is an $\alpha$-derivation of hom-Lie-Rinehart algebra $(\mathcal{L},\alpha_L)$ with the symbol $\sigma_D$.
\end{Exm}

\begin{Prop}\label{RemonDer1}
Let $(\mathcal{N},\alpha_N)\xrightarrow{j} (\mathcal{K} ,\alpha_{K}) \xrightarrow{f}(\mathcal{L},\alpha_L)$ be a central extension of hom-Lie-Rinehart algebra. If $D$ and $D^{\prime}$ are $\alpha$-derivations of $(\mathcal{K},\alpha_K)$ with the same symbol $\sigma_D$ such that $f\circ D=f\circ D^{\prime}$, then $D|_{\{K,K\}}=D^{\prime}|_{\{K,K\}}$. 
\end{Prop}
\begin{proof}
Since $f\circ D=f\circ D^{\prime}$, for any $x\in K$ we have $D(x)-D^{\prime}(x)\in Ker(f)\subset Z_A(\mathcal{K},\alpha_K)$. Thus for any $x,y\in K$ and $a\in A$, we have $a[D(x),y]=a[D^{\prime}(x),y]$. Next, let $x\in \{K,K\}$, i.e. $x=\sum a_i[x_{i_1},x_{i_2}]$, then we get the following:
\begin{align*}
D(x)&=D(\sum_i a_i[x_{i_1},x_{i_2}])\\
&=\sum_i \phi(a_i).D([x_{i_1},x_{i_2}])+\sigma_D(a).\alpha_K([x_{i_1},x_{i_2}])\\
&=\sum_i \phi(a_i).\big([D(x_{i_1}),\alpha_K(x_{i_2})])+[\alpha_K(x_{i_1}),D(x_{i_2})]\big)+\sigma_D(a).\alpha_K([x_{i_1},x_{i_2}])\\
&=D^{\prime}(\sum_i a_i[x_{i_1},x_{i_2}])=D^{\prime}(x).
\end{align*} 
Hence, $D|_{\{K,K\}}=D^{\prime}|_{\{K,K\}}$.
\end{proof}

\begin{Prop}
Let $D$ be an $\alpha$-derivation of hom-Lie-Rinehart algebra $(\mathcal{L},\alpha_L)$ with symbol $\sigma_D$. Define a map $\mathfrak{uce}_A^{\phi}(D):\mathfrak{uce}_A^{\phi}(L)\rightarrow \mathfrak{uce}_A^{\phi}(L)$ as follows:
$$\mathfrak{uce}_A^{\phi}(D)(a,x,y)=(\sigma_D(a),\alpha_L(x),\alpha_L(y))+(\phi(a),D(x),\alpha_L(y))+(\phi(a),\alpha_L(x),D(y))$$
for any $(a,x,y)\in \mathfrak{uce}_A^{\phi}(L)$. Then the map $\mathfrak{uce}_A^{\phi}(D)$ is an $\alpha$-derivation of hom-Lie-Rinehart algebra $\mathfrak{uce}_A^{\phi}(\mathcal{L},\alpha_L)$ with symbol $\sigma_D$. Moreover, recall the homomorphism $u:\mathfrak{uce}_A^{\phi}(\mathcal{L},\alpha_L)\rightarrow (\mathcal{L},\alpha_L)$ defined as $u(a,x,y)=a[x,y]_L$. Then $\mathfrak{uce}_A^{\phi}(D)(Ker(u))\subset Ker(u)$.  

\end{Prop}
\begin{proof}
By using the definition of the hom-Lie-Rinehart algebra $\mathfrak{uce}_A^{\phi}(\mathcal{L},\alpha_L)$ and the map $\mathfrak{uce}_A^{\phi}(D)$, it easily follows that $\mathfrak{uce}_A^{\phi}(D)$ is an $\alpha$-derivation of hom-Lie-Rinehart algebra $\mathfrak{uce}_A^{\phi}(\mathcal{L},\alpha_L)$ with symbol $\sigma_D$. Let us denote $\mathfrak{uce}_A^{\phi}(D):=D^u$. Next, let $(a,x,y)\in Ker(u)$, i.e. $a[x,y]=0$, which implies $D(a[x,y])=0$.
\begin{align*}
u(D^u(a,x,y))&=u((\sigma_D(a),\alpha_L(x),\alpha_L(y))+(\phi(a),D(x),\alpha_L(y))+(\phi(a),\alpha_L(x),D(y)))\\
&=\sigma_D(a).[\alpha_L(x),\alpha_L(y)]+\phi(a).[D(x),\alpha_L(y)]+\phi(a).[\alpha_L(x),D(y)]\\
&=D(a.[x,y])\\
&=0.
\end{align*}
Thus, $\mathfrak{uce}_A^{\phi}(D)(Ker(u)):=D^u(Ker(u))\subset Ker(u)$.
\end{proof}

\begin{Rem}\label{RemDer2}
Let $f:(\mathcal{K},\alpha_K)\rightarrow (\mathcal{L},\alpha_L)$ be a homomorphism in the category $hLR_A^{\phi}$. If $D_K$ is an $\alpha$-derivation of $(\mathcal{K},\alpha_K)$, and $D_L$ is an $\alpha$-derivation of $(\mathcal{K},\alpha_K)$ with same symbols $\sigma_D\in Der_{\phi}(A)$ such that   
$f\circ D_K=D_L\circ f$, then $\mathfrak{uce}_A^{\phi}(f)\circ D_K^u=D_L^u\circ\mathfrak{uce}_A^{\phi}(f)$.

\end{Rem}

\begin{Thm}\label{Lift2}
Let $$
(\mathcal{N},\alpha_N)\xrightarrow{j} (\mathcal{K} ,\alpha_{K}) \xrightarrow{f}(\mathcal{L},\alpha_L) $$
be a central extension, where $(\mathcal{K} ,\alpha_{K})$ is an $\alpha$-perfect hom-Lie-Rinehart algebra. An $\alpha$-derivation $D_L$ of $(\mathcal{L},\alpha_L)$ (with symbol $\sigma_D\in Der_{\phi}(A)$) lifts to a unique $\alpha$-derivation $D_K$ of $(\mathcal{K},\alpha_K)$ (with same symbol $\sigma_D\in Der_{\phi}(A)$) satisfying 
$D_L f=f D_K$ if and only if $D_L^u(P)\subset P$. Moreover, $Ker(f)$ is invariant under the $\alpha$-derivation $D_K$, i.e. $D_K(Ker(f))\subset Ker(f)$. 
\end{Thm}

\begin{proof}
Let us assume that a lift $D_K$ of the $\alpha$-derivation $D_L$ exists such that $D_L f=f D_K$, then by Proposition \ref{RemonDer1}, and the perfectness of the hom-Lie-Rinehart algebra $(\mathcal{K},\alpha_K)$ this lift $D_K$ is unique. Next,
\begin{align*}
D_L^u(P)&=D_L^u(\mathfrak{uce}^{\phi}_A(f)(Ker(u_K)))\\
&=(D_L^u\circ\mathfrak{uce}^{\phi}_A(f))(Ker(u_K))\\
&=(\mathfrak{uce}^{\phi}_A(f)\circ D_K^u)(Ker(u_K))\\
&\subset \mathfrak{uce}^{\phi}_A(f)(Ker(u_K))=P.
\end{align*} 
Here, we used the commutativity: $D_L^u\circ\mathfrak{uce}^{\phi}_A(f)=\mathfrak{uce}^{\phi}_A(f)\circ D_K^u$ from Remark \ref{RemDer2} and the fact that $D_K^u(Ker(u_K))\subset Ker(u_K)$. 

Conversely, let $D_L$ is an $\alpha$-derivation of $(\mathcal{L},\alpha_L)$ with the symbol $\sigma_D\in Der_{\phi}(A)$ such that $D_L^u(P)\subset P$.
Also, consider the central extension: 
$$(\mathcal{P},\alpha_P)\xrightarrow{k}\mathfrak{uce}^{\phi}_A(\mathcal{L},\alpha_{L})\xrightarrow{u_K\circ \mathfrak{uce}_A^{\phi}(f)^{-1}}(\mathcal{K},\alpha_K).$$
Next, let us define a map $D_K:K\rightarrow K$ as follows:
$$D_K(k)=u_K\circ \mathfrak{uce}_A^{\phi}(f)^{-1}(\mathfrak{uce}_A^{\phi}(D_L)(\bar{k}))$$
for any $k\in K$. Here $\bar{k}\in \mathfrak{uce}_A^{\phi}L$ such that $u_K\circ \mathfrak{uce}_A^{\phi}(f)^{-1}(\bar{k})=k$. Then, by the definition of the map $D_K$ and the fact that $D_L^u(P)\subset P$, it easily follows that $D_K$ is a well-defined $\alpha$-derivation of $(\mathcal{K},\alpha_K)$ with symbol $\sigma_D$. Let $k\in Ker(f)$, then 
$$f(D_K(k))=D_L(f(k))=0,$$
i.e. $D_K(Ker(f))\subset Ker(f)$. 
\end{proof}
\begin{Rem}
In particular, if $A=R$ and $\phi=Id_A$, then any hom-Lie-Rinehart algebra $(\mathcal{L},\alpha_L)$ over $(A,\phi)$ is just a hom-Lie $R$-algebra $(L,[-,-]_L,\alpha_L)$. Then the above Theorem \ref{Lift1} and Theorem \ref{Lift2} answer the question: when we can lift automorphisms and $\alpha$-derivations (\cite{Sheng12}) of $\alpha$-perfect hom-Lie algebras to central extensions. In fact, using the same proof and techniques one can get the results for lifting of $\alpha^k$-derivations (\cite{Sheng12}) of $\alpha$-perfect hom-Lie algebras. 
\end{Rem}
\section{Non-abelian Tensor product of hom-Lie-Rinehart algebras}
 In this section, first we define quasi hom-actions of a hom-Lie-Rinehart algebra on another hom-Lie-Rinehart algebra in the category $hLR_A^{\phi}$. Then we introduce a non-abelian tensor product of hom-Lie-Rinehart algebras in $hLR_A^{\phi}$ and discuss its properties. Finally, we express the universal central extension (and the universal $\alpha$-central extension) of a hom-Lie-Rinehart algebra in terms of this non-abelian tensor product.
  
\begin{Def}
Let $(\mathcal{L},\alpha_L)$ and $(\mathcal{M},\alpha_M)$ be hom-Lie-Rinehart algebras over $(A,\phi)$. A map $\theta :L\times M\rightarrow M$ given by $\theta(x,m)=x_m$ is called a {\bf quasi hom-action} of $(\mathcal{L},\alpha_L)$ on $(\mathcal{M},\alpha_M)$ if the following hold:

\begin{enumerate}
\item $\prescript{x}{}{(a.m)}=\phi(a).\prescript{x}{}{m}+x(a).\alpha_{M}(m);$
\item $\prescript{[x,y]}{}{\alpha_M(m)}=\prescript{\alpha_L(x)}{}{(\prescript{y}{}{m})}-\prescript{\alpha_L(y)}{}{(\prescript{x}{}{m})},$
\item $\prescript{\alpha_L(x)}{}{[m,n]}=[\prescript{x}{}{m},\alpha_M(n)]+[\alpha_M(m),\prescript{x}{}{n}],$
\item $\prescript{x}{}{m}(\phi(a))=\alpha_L(x)(m(a)
)-\alpha_M(m)(x(a)),$
\item $\alpha_M(\prescript{x}{}{m})=\prescript{\alpha_L(x)}{}{\alpha_M(m)},$
\end{enumerate}
for all $x,y\in L$, $m,n\in M$, and $a\in A$.
\end{Def}

In particular, if $A=R$, and $\phi=Id_A$, the above definition gives the hom-action of a hom-Lie algebra $(L,[-,-]_L,\alpha_L)$ on hom-Lie algebra $(M,[-,-]_M,\alpha_M)$, as defined in \cite{CKP17}.

\begin{Exm}
Let $(\mathcal{L},\alpha_L)$ be a hom-Lie-Rinehart algebra over $(A,\phi)$, then the left action of $(\mathcal{L},\alpha_L)$ on itself by the underlying hom-Lie bracket: $(x,y)\mapsto [x,y]_L$ is a quasi hom action of $(\mathcal{L},\alpha_L)$ on itself. 

\end{Exm}
Let $(\mathcal{L},\alpha_L)$ and $(\mathcal{M},\alpha_M)$ be hom-Lie-Rinehart algebras over $(A,\phi)$, which have quasi hom-actions on each other, then these 
two quasi hom-actions $L\times M\rightarrow M$; $(x,m)\mapsto \prescript{x}{}{m}$ and $M\times L\rightarrow L$; $(m,x)\mapsto \prescript{m}{}{x}$ are said to be compatible if 

\begin{enumerate}
\item $\prescript{x}{}{m}(a)=-\prescript{m}{}{x}(a),$
\item $\prescript{\prescript{m}{}{x}}{}{n}=[n,\prescript{x}{}{m}],$
\item $\prescript{\prescript{x}{}{m}}{}{y}=[y,\prescript{m}{}{x}],$
\end{enumerate}
for all $x,y\in L$, $m,n\in M$, and $a\in A$. 

If $(\mathcal{L},\alpha_L)$ is a hom-Lie-Rinehart algebra (over $(A,\phi)$), and  pairs $(\mathcal{N},\alpha_N)$ and $(\mathcal{P},\alpha_P)$ are two quasi ideals of $(\mathcal{L},\alpha_L)$ then the quasi hom-actions of these quasi ideals on each other, given by the underlying hom-Lie bracket are compatible.

\begin{Def}
Let $(\mathcal{L},\alpha_L)$ and $(\mathcal{M},\alpha_M)$ be two hom-Lie-Rinehart algebras over $(A,\phi)$ and $(x,m)\mapsto \prescript{x}{}{m}$, $(m,x)\mapsto \prescript{m}{}{x}$ defines quasi hom-action of these two on each other. Let $(\mathcal{N},\alpha_N)$ hom-Lie-Rinehart algebras over $(A,\phi)$ then a bilinear map $f:L\times M\rightarrow N$ is called a {\bf hom-Lie-Rinehart pairing} if the following conditions hold.
\begin{enumerate}
\item $f(x,m)(a)=(\prescript{x}{}{m})(a),$
\item $f([x,y],\alpha_M(m))=f(\alpha_L(x),\prescript{y}{}{m})-f(\alpha_{L}(y),\prescript{x}{}{m}),$
\item $f(\alpha_L(x),[m,n])=f(\prescript{n}{}{x},\alpha_M(m))-f(\prescript{m}{}{x},\alpha_M(n)),$
\item $f(\alpha_L(x),\alpha_M(m))=\alpha_N(f(x,m)),$
\item $f(a.\prescript{m}{}{x},b.\prescript{y}{}{n})=-\phi(ab).[f(x,m),f(y,n)]-\phi(a)f(x,m)(b).\alpha_N(f(y,n))+\phi(b)f(y,n)(a).\alpha_N(f(x,m))$
\end{enumerate} 
for all $x,y\in L$, $a,b\in A$, and $m,n\in M$. A hom-Lie-Rinehart pairing $f:L\times M\rightarrow N$ is called universal if for any other hom-Lie-Rinehart pairing $g:L\times M\rightarrow N^{\prime}$, there exists a unique morphism $h:(\mathcal{N},\alpha_N)\rightarrow (\mathcal{N}^{\prime},\alpha_{N^{\prime}})$ in $hLR_A^{\phi}$ such that $h\circ f=g$.
\begin{Exm}
Let $(\mathcal{L},\alpha_L)$ is a hom-Lie-Rinehart algebra and $(\mathcal{N},\alpha_N)$ and $(\mathcal{P},\alpha_P)$ are two quasi ideals of $(\mathcal{L},\alpha_L)$ then the map $f:N\times P\rightarrow N\cap P$ defined by $f(n,p)=[n,p]$, is a hom-Lie-Rinehart pairing. 
\end{Exm}
\end{Def}
Let $(\mathcal{L},\alpha_L), (\mathcal{M},\alpha_M)\in hLR_A^{\phi}$, and let both of them have compatible quasi hom-actions on each-other, then let us define an $A$-module: $L\ast M$, spanned by the symbols $x\ast m$ satisfying:
\begin{enumerate}
\item $(x+ky)\ast m=x\ast m+k(y\ast m),$
\item $x\ast (m+kn)=x\ast m+k(x\ast n),$
\item $[x,y]\ast \alpha_M(m)=\alpha_L(x)\ast \prescript{y}{}{m}-\alpha_L(y)\ast \prescript{x}{}{m},$
\item $\alpha_L(x)\ast [m,n]=\prescript{n}{}{x}\ast \alpha_M(m)-\prescript{m}{}{x}\ast \alpha_M(n),$
\item $a(\prescript{m}{}{x})\ast b(\prescript{y}{}{n})=-b(\prescript{n}{}{y})\ast a(\prescript{x}{}{m}),$
\item $a(\prescript{m}{}{x})\ast b(\prescript{y}{}{n})=\phi(ab).(\prescript{m}{}{x}\ast\prescript{y}{}{n})-\phi(a)\prescript{x}{}{m}(b).(\alpha_L(y)\ast \alpha_M(n))$\\
+$\phi(b)\prescript{y}{}{n}(a).(\alpha_L(x)\ast \alpha_M(m)),$
\end{enumerate} 
for all $x,y\in L$, $m,n\in M$, $k\in R$, and $a,b\in A$. 

\begin{Prop}\label{NON-AB}
 The tuple $(\mathcal{L}\ast \mathcal{M},\alpha_{L\ast M}):=(A,L\ast M,\phi,[-,-]_{L\ast M},\alpha_{L\ast M},\rho_{L\ast M})$ is a hom-Lie-Rinehart algebra over $(A,\phi)$, where:
 \begin{itemize}
 \item  The endomorphism $\alpha_{L\ast M}:L\ast M\rightarrow L\ast M$ is defined by: $$\alpha_{L\ast M}(x\ast m)=\alpha_L(x)\ast \alpha_M(m),$$
 \item Hom-Lie bracket $[-,-]_{L\ast M}$ is given by: 
 $$[a(x\ast m),b(y\ast n)]=-a(\prescript{m}{}{x})\ast b(\prescript{y}{}{n}),$$
 \item The anchor map $\rho_{L\ast M}:L\ast M\rightarrow Der_{\phi}(A)$ is defined by: $$\rho_{L\ast M}(x\ast m)(a)=\prescript{x}{}{m}(a)$$

\end{itemize}  
for any $x,y\in L$, $m,n\in M$, and $a\in A$.
\end{Prop}

\begin{Def}\label{Def-Non-abelian}
Let $(\mathcal{L},\alpha_L)$, and $ (\mathcal{M},\alpha_M)$ be hom-Lie-Rinehart algebras over $(A,\phi)$, which have compatible quasi hom-actions on each-other. Then the hom-Lie-Rinehart algebra $(\mathcal{L}\ast \mathcal{M},\alpha_{L\ast M})$, given by Proposition \ref{NON-AB}, is called the {\bf non-abelian tensor product} of hom-Lie-Rinehart algebras $(\mathcal{L},\alpha_L)$, and $(\mathcal{M},\alpha_M)$.

\end{Def}

In particular, if $A=R$, and $\phi=Id_A$, the above definition gives the non-abelian tensor product of hom-Lie algebras $(L,[-,-]_L,\alpha_L)$ and $(M,[-,-]_M,\alpha_M)$, defined in \cite{CKP17}.

\begin{Rem}\label{universal-property}
Note that the map $f:L\times M\rightarrow L\ast M$ defined by $f(x,m)=x\ast m$ is a hom-Lie-Rinehart pairing. Let $(\mathcal{N},\alpha_N)$ be a hom-Lie-Rinehart algebra and $g:L\times M\rightarrow N$ be a hom-Lie-Rinehart pairing then define a map
$$\Phi:(\mathcal{L}\ast \mathcal{M},\alpha_{L\ast M})\rightarrow (\mathcal{N},\alpha_N)$$ 
by $\Phi(x\ast m)=g(x,m)$, then $\Phi$ is a hom-Lie-Rinehart algebra homomorphism since $g$ is a hom-Lie-Rinehart pairing. Also, by definition of the map $\Phi$, we have $\Phi\circ f=g$. Therefore, the pairing $f:L\times M\rightarrow L\ast M$ is a universal hom-Lie-Rinehart pairing. 
\end{Rem}

\begin{Rem}\label{symmetry}
Consider the hom-Lie-Rinehart algebra $(\mathcal{M}\ast\mathcal{L},\alpha_{M\ast L})$. Define a map $g:L\times M\rightarrow M\ast L$ by $g(x,m)=m\ast x$. It easily follows that $f$ is a universal hom-Lie-Rinehart pairing. Hence, by using the universal property of hom-Lie-Rinehart pairing and Remark \ref{universal-property}, we have an isomorphism of hom-Lie-Rinehart algebras:
$$(\mathcal{L}\ast\mathcal{M},\alpha_{L\ast M})\cong(\mathcal{M}\ast\mathcal{L},\alpha_{M\ast L})$$
\end{Rem}

Now, we have the following result, which easily follows from the Definition \ref{Def-Non-abelian} of non-abelian tensor product:
\begin{Prop}\label{HOM-MORPH}
There exist hom-Lie-Rinehart algebra homomorphisms
$$\pi_1:(\mathcal{L}\ast \mathcal{M},\alpha_{L\ast M})\rightarrow (\mathcal{L},\alpha_L),~~\mbox{and}~~ \pi_2:(\mathcal{L}\ast \mathcal{M},\alpha_{L\ast M})\rightarrow (\mathcal{M},\alpha_M)$$
defined by $\pi_1(a.(x\ast m))=-a(\prescript{m}{}{x})$ and $\pi_2(a.(x\ast m))=a(\prescript{x}{}{m})$, for any $x\in L, ~m\in M,$ and $a\in A$. Moreover, $ker(\pi_1)\subset Z_A(\mathcal{L}\ast \mathcal{M})$ and $ker(\pi_2)\subset Z_A(\mathcal{L}\ast \mathcal{M})$.
\end{Prop}

Let $(\mathcal{L},\alpha_L)$, and $ (\mathcal{M},\alpha_M)$ be hom-Lie-Rinehart algebras over $(A,\phi)$ with trivial quasi hom-actions on each other. If $\alpha_M$ and $\alpha_L$ are surjective maps, then by definition of non-abelian tensor product, for any $x,y\in L$, and $m,n\in M$ we have the following:
$$[x,y]\ast m=0=x\ast [m,n].$$

Subsequently, we have an isomorphism of $A$-modules: $\Phi:L\ast M\rightarrow L^{ab}\otimes M^{ab}$, defined by $\Phi(x\ast m)=x\otimes m$ for any $x\in L, m\in M$.  It follows that the map $\Phi$ gives an isomorphism of hom-Lie-Rinehart algebras:
$$(\mathcal{L}\ast \mathcal{M},\alpha_{L\ast M} )\cong \big(\mathcal{L}^{ab}\otimes \mathcal{M}^{ab},\alpha_{L^{ab}\otimes M^{ab}}\big),$$
where
$\big(\mathcal{L}^{ab}\otimes \mathcal{M}^{ab},\alpha_{L^{ab}\otimes M^{ab}}\big)$ is a hom-Lie-Rinehart algebra with underlying $A$-module: $L^{ab}\otimes M^{ab}$ ($L^{ab}=\frac{L}{[L,L]},$ and $M^{ab}=\frac{M}{[M,M]}$ are abelianizations of hom-Lie algebras), trivial hom-Lie bracket, trivial anchor map, and the map $\alpha_{L^{ab}\otimes M^{ab}}$ is induced by maps $\alpha_L$ and $\alpha_M$. 

Let $\sigma:(\mathcal{L}_1,\alpha_{L_1})\rightarrow (\mathcal{L}_2,\alpha_{L_2})$ and $\tau: (\mathcal{M}_1,\alpha_{M_1})\rightarrow (\mathcal{M}_2,\alpha_{M_2})$ be homomorphisms in the category $hLR_A^{\phi}$. Also assume that hom-Lie-Rinehart algebras $(\mathcal{L}_i,\alpha_{L_i})$
 and $(\mathcal{M}_i,\alpha_{M_i})$ (for $i=1,2$) have compatible quasi hom-actions on each other. Then we say that homomorphisms $\sigma,$ and $\tau$ preserve these quasi hom-actions if $\sigma,\tau$ satisfy the following condition: 
 \begin{equation}\label{pres}
 \sigma(\prescript{m}{}{x})=\prescript{\tau(m)}{}{\sigma(x)},~ \mbox{and}~ \tau(\prescript{x}{}{m})=\prescript{ \sigma(x)}{}{\tau(m)}, ~~\mbox{for}~x\in L_1, m\in M_1.
 \end{equation} 
 
 Next, let $\sigma,$ and $\tau$ preserve these quasi hom-actions. Define an $A$-linear map $\sigma\ast \tau: L_1\ast M_1\rightarrow L_2\ast M_2$, mapping $x\ast m\mapsto \sigma(x)\ast \tau(m)$ for $x\in L_1$, $m\in M_1$. Then, by using equation \eqref{pres} it follows that the map
 $$\sigma\ast \tau:(\mathcal{L}_1\ast \mathcal{M}_1,\alpha_{L_1\ast M_1} )\rightarrow (\mathcal{L}_2\ast \mathcal{M}_2,\alpha_{L_2\ast M_2} )$$
is a homomorphism in the category $hLR_A^{\phi}$. 
\begin{Prop}
Let $(\mathcal{L}_1,\alpha_{L_1})\xrightarrow{f}(\mathcal{L}_2,\alpha_{L_2})\xrightarrow{g}(\mathcal{L}_3,\alpha_{L_3})$ be a short exact sequence in the category $hLR_A^{\phi}$ and $(\mathcal{M},\alpha_M)$ be a hom-Lie-Rinehart algebra over $(A,\phi)$. If for each $i\in \{1,2,3\}$, hom-Lie-Rinehart algebras $(\mathcal{M},\alpha_M)$ and $(\mathcal{L}_i,\alpha_{L_i})$ have compatible quasi hom-actions on each other and morphisms $f$ and $g$ satisfy the following conditions: $f(\prescript{p}{}{x})=\prescript{p}{}{f(x)}, ~\prescript{x}{}{p}=\prescript{f(x)}{}{p},~g(\prescript{p}{}{s})=\prescript{p}{}{g(s)},$ and $\prescript{s}{}{p}=\prescript{g(s)}{}{p}$, then  
$$(\mathcal{L}_1\ast \mathcal{M},\alpha_{L_1\ast M})\xrightarrow{f\ast Id_M}(\mathcal{L}_2\ast \mathcal{M},\alpha_{L_2\ast M})\xrightarrow{g\ast Id_M}(\mathcal{L}_3\ast \mathcal{M},\alpha_{L_3\ast M})$$
is an exact sequence and the map $g\ast Id_M$ is a surjective map.  
\end{Prop}
\begin{proof}
First note that the maps $f\ast Id_M$ and $g\ast Id_M$ are morphisms in $hLR_A^{\phi}$. The map $g$ is surjective implies that $g\ast id_M$ is a surjective map and $Im(f\ast Id_M)\subset Ker(g\ast Id_M)$ since $Im(f)=Ker(g)$. To complete the proof we need to show that $Ker(g\ast Id_M)\subset Im(f\ast Id_M)$.

Now for all $x\in L_1,~ m\in M$ and $a\in A$ we can deduce the following equations. 
\begin{enumerate}
\item $0=x\ast m(a)=\prescript{x}{}{m}(a)=\prescript{f(x)}{}{m}(a)=f(x)\ast m(a)=(f\ast Id_M(x\ast m))(a),$
\item $[f(x)\ast m, y\ast n]=\prescript{m}{}{f(x)}\ast \prescript{y}{}{n}=f(\prescript{m}{}{x})\ast \prescript{y}{}{n}=f\ast Id_M(\prescript{m}{}{x}\ast \prescript{y}{}{n})$.
\end{enumerate} 
 So, $Im(f\ast Id_M)$ is an ideal of the hom-Lie-Rinehart algebra $(\mathcal{L}_2\ast \mathcal{M},\alpha_{L_2\ast M})$. We consider the quotient hom-Lie-Rinehart algebra $(\mathcal{L}_2\ast \mathcal{M}/Im(f\ast Id_M),\bar{\alpha}_{L_2\ast M})$, where the underlying $A$-module is the quotient module $L_2\ast M/Im(f\ast Id_M)$ and $\bar{\alpha}_{L_2\ast M}$ is the induced linear map.
Since $Im(f\ast Id_M)\subset Ker(g\ast Id_M)$, we have canonical surjective homomorphism 
$$\Psi:(\mathcal{L}_2\ast \mathcal{M}/Im(f\ast Id_M),\bar{\alpha}_{L_2\ast M})\rightarrow (\mathcal{L}_3\ast \mathcal{M},\alpha_{L_3\ast M}).$$
Let us define a map 
$\phi:L_3\times M \rightarrow (L_2\ast M)/ Im(f\ast Id_M)$ given by $$\phi(t,m)=p\ast m+Im(f\ast Id_M)$$ for any $t\in L_3,~p\in L_2$ where $g(p)=t$ and $m\in M$. Then it follows that the map $\phi$ is a hom-Lie-Rinehart pairing and hence by universal property of non-abelian tensor product we get a unique morphism in $hLR_A^{\phi}$:
$$\Phi: (\mathcal{L}_3\ast \mathcal{M},\alpha_{L_3\ast M})\rightarrow (\mathcal{L}_2\ast \mathcal{M}/Im(f\ast Id_M),\bar{\alpha}_{L_2\ast M}).$$ 
Here $\Phi \circ \Psi$ is the identity map. This implies that $Ker(g\ast Id_M)\subset Im(f\ast Id_M)$. 

\end{proof}

Let $(\mathcal{L},\alpha_L)$ be a perfect hom-Lie-Rinehart algebra over $(A,\phi)$ and it has a quasi hom-action on itself by the underlying hom-Lie bracket, i.e. $\prescript{x}{}{y}=[x,y]_L$. Then from Proposition \ref{HOM-MORPH}, it follows that the short exact sequence 
\begin{equation*}
(\mathcal{P},\alpha_P)\xrightarrow{i}(\mathcal{L}\ast \mathcal{L},\alpha_{L\ast L})\xrightarrow{\pi} (\mathcal{L},\alpha_L)
\end{equation*}      
is a central extension where $\pi(a(x\ast y))=a[x,y]_L$. The underlying $A$-module $P:=Ker(\pi)$, and the map $\alpha_P=\alpha_{L\ast L}\big|_{P}$ in the hom-Lie-Rinehart algebra $(\mathcal{P},\alpha_P)$. 

\begin{Thm}
The central extension
\begin{equation*}
(\mathcal{P},\alpha_P)\xrightarrow{i}(\mathcal{L}\ast \mathcal{L},\alpha_{L\ast L})\xrightarrow{\pi} (\mathcal{L},\alpha_L)
\end{equation*}      
is a universal central extension of $(\mathcal{L},\alpha_L)$. Moreover, if $(\mathcal{L},\alpha_L)$ is an $\alpha$-perfect hom-Lie-Rinehart algebra, then the above central extension is a universal $\alpha$-central extension of $(\mathcal{L},\alpha_L)$.
\begin{proof}
Let $p:(\mathcal{M},\alpha_M)\rightarrow (\mathcal{L},\alpha_L)$ be a central extension of $(\mathcal{L},\alpha_L)$ and $s:L\rightarrow M$ be a map with $p\circ s=Id_L$. Then we have the following observations.
\begin{enumerate}
\item $s[x,y]_L-[s(x),s(y)]_M\in Ker(p),$
\item $s(a.x)-a.s(x)\in ker(p),$
\item $s(\alpha_L(x))-\alpha_M(s(x))\in Ker(p),$
\item $s(x)(a)=p(s(x))(a)=x(a),$
\end{enumerate}
for any $x,y\in L,~a\in A$. Now, let us define a map $q:L\times L\rightarrow M $ by $q(x,y)=[s(x),s(y)]_{M}$. Then by the above observations and the fact that $ker(p)\subset Z_A(\mathcal{M})$, 
 the map $q:L\times L\rightarrow M$ is a hom-Lie-Rinehart pairing and therefore it extends to a hom-Lie-Rinehart algebra homomorphism $\mathfrak{q}:(\mathcal{L}\ast \mathcal{L},\alpha_L\ast \alpha_L)\rightarrow (\mathcal{M},\alpha_M)$. (Here,  $$\mathfrak{q}(\alpha_{L\ast L}(x\ast y))=q(\alpha_L(x),\alpha_L(y))=\alpha_M(q(x,y))=\alpha_M(\mathfrak{q}(x\ast y))$$
i.e. $\mathfrak{q}\circ \alpha_{L\otimes L}=\alpha_M\circ \mathfrak{q}$). It easily follows from definition of the map $q$ that $p\circ \mathfrak{q}=\pi$. By Lemma \ref{uniq}, it follows that the homomorphism $\mathfrak{q}$ is unique. Hence, the non-abelian tensor product $(\mathcal{L}\ast \mathcal{L},\alpha_{L\ast L})$ is a universal central extension of  $(\mathcal{L},\alpha_L)$.

Next, let $(\mathcal{L},\alpha_L)$ is an $\alpha$-perfect hom-Lie-Rinehart algebra and $$(\mathcal{N},\alpha_N)\xrightarrow{i}(\mathcal{M},\alpha_M)\xrightarrow{p}(\mathcal{L},\alpha_L)$$ be an $\alpha$-central extension of $(\mathcal{L},\alpha_L)$. Consider a map $s:L\rightarrow M$ such that $p\circ s=Id_L$ and define a map $f:L\times L\rightarrow M$ by 
$$f(\alpha_L(x),\alpha_L(y))=[\alpha_M(s(x)),\alpha_M(s(y))]$$
for all $x,y\in L$. It follows that $f$ is a hom-Lie-Rinehart pairing and therefore the map $f$ extends to a homomorphism $\tilde{f}:(\mathcal{L}\ast \mathcal{L},\alpha_{L\ast L})\rightarrow (\mathcal{M},\alpha_M)$. Since $(\mathcal{L},\alpha_L)$ is $\alpha$-perfect, by Lemma \ref{uniq2}, the map $\tilde{f}$ is unique. Hence, the central extension 
\begin{equation*}
(\mathcal{P},\alpha_P)\xrightarrow{i}(\mathcal{L}\ast \mathcal{L},\alpha_{L\ast L})\xrightarrow{\pi} (\mathcal{L},\alpha_L)
\end{equation*}      
is a universal $\alpha$-central extension of $(\mathcal{L},\alpha_L)$.

 

\end{proof}
\end{Thm}

\end{document}